\documentclass[12pt]{amsart}
\usepackage{amsthm}
\usepackage{amscd}
\usepackage{url}
\usepackage{graphicx}
\usepackage{tikz-cd}
\usepackage{mathtools}
\usepackage{color}
\usepackage{amsmath}
\usepackage{amssymb}
\usepackage{amsthm}
\usepackage{amscd}
\usepackage{eucal} 
\usepackage{enumitem}
\usepackage{mathdots}
\usepackage{tikz}
\usepackage{hyperref}
\usepackage[all]{xy}

\allowdisplaybreaks

\newtheorem{theorem}[equation]{Theorem}

\newtheorem{lemma}[equation]{Lemma}

\newtheorem{proposition}[equation]{Proposition}
\newtheorem{corollary}[equation]{Corollary}

\newtheorem*{theorem:derhamisomorphism}{Theorem~\ref{T:derhamisomorphism}}
\newtheorem*{theorem:characterization}{Theorem~\ref{T:characterization}}

\theoremstyle{definition}
\newtheorem{definition}[equation]{Definition}

\newtheorem{remark}[equation]{Remark}

\numberwithin{equation}{section}

\newcommand{\F}{\mathbb{F}}

\newcommand{\Z}{\mathbb{Z}}
\newcommand{\QQ}{\mathbb{Q}}

\newcommand{\HH}{\mathbb{H}}
\newcommand{\TT}{\mathbb{T}}

\newcommand{\C}{\mathbb{C}}
\newcommand{\CC}{\mathbb{C}}

\newcommand{\bb}{\mathbf{b}}

\newcommand{\bm}{\mathbf{m}}
\newcommand{\bn}{\mathbf{n}}

\newcommand{\bw}{\mathbf{w}}
\newcommand{\bz}{\mathbf{z}}

\newcommand{\zz}{\mathbf{z}}

\newcommand{\ww}{\mathbf{w}}

\DeclareMathOperator{\Log}{Log}
\DeclareMathOperator{\Ker}{Ker}
\DeclareMathOperator{\GL}{GL}
\DeclareMathOperator{\Frac}{Frac}
\DeclareMathOperator{\Mat}{Mat}
\DeclareMathOperator{\Frob}{Frob}

\DeclareMathOperator{\Exp}{Exp}

\DeclareMathOperator{\dmat}{d}

\DeclareMathOperator{\Res}{Res}

\DeclareMathOperator{\Id}{Id}

\DeclareMathOperator{\Cof}{Cof}

\DeclareMathOperator{\Lie}{Lie}

\newcommand{\tr}{\mathrm{tr}}
\newcommand{\tpi}{\widetilde{\pi}}

\newcommand{\inv}{\ensuremath ^{-1}}
\newcommand{\twist}{^{(1)}}

\newcommand{\twistk}[1]{^{(#1)}}
\newcommand{\bomega}{\boldsymbol{\omega}}

\newcommand{\inorm}[1]{{\lvert #1 \rvert}}

\definecolor{ForestGreen}{rgb}{0.0, 0.5, 0.0}

\begin{document}

\title[Mellin Transform Formulas]{Mellin Transform Formulas for Anderson Modules and Hints of Modularity}

	\author{O\u{g}uz Gezm\.{i}\c{s}}
\address{Department of Mathematics, National Tsing Hua University, Hsinchu City 30042, Taiwan R.O.C.}
\email{gezmis@math.nthu.edu.tw}


\author{Nathan Green}

\address{Department of Mathematics and Statistics, University of North Carolina at Charlotte}
\email{ngreen50@charlotte.edu}

\date{\today}

\keywords{Drinfeld modules, $t$-motives, Mellin transform, Drinfeld modular forms, special values of $L$-functions}

	\subjclass[2010]{Primary 11G09, 11M38}

\maketitle
\begin{abstract} In the present paper, we introduce formulas for the logarithm function of abelian, uniformizable Anderson modules. Combining our formulas with a motivic map introduced recently by the second author, we relate special values of dual Goss $L$-functions of Drinfeld modules to the rigid analytic trivialization of their corresponding  $t$-comotives. We apply these formulas to relate certain $L$-values to special values of matrix valued modular forms, generalizing the vector valued Drinfeld modular forms of Pellarin and Pellarin and Perkins in the rank two setting. 
\end{abstract}


\section{Introduction}

\subsection{Background and Context}
The classical modularity theorem \cite{Wil95} establishes a deep connection between elliptic curves and modular forms. The proof of that theorem was a massive undertaking involving dozens of mathematicians and necessitating the development of novel techniques in a wide array of mathematics, including modular curves, Galois representations, Iwasawa theory, algebraic geometry and much more. However, nearly 50 years before this remarkable achievement, mathematicians began suspecting that such a connection should exist based on relatively straightforward analytic arguments. Namely, given a weight 2 Hecke eigenform $f(z)$ with integral Fourier series
\begin{equation}
f(z) = \sum_{n=1}^\infty a_n q^n,\quad q = e^{2\pi i z},
\end{equation}  
the $L$-function $L_f(z):=\sum_{n=1}^\infty a_n n^{-z}$ associated to $f$ is the $L$-function of an elliptic curve defined over $\mathbb{Q}$. Moreover, Shimura gave a concrete description of the elliptic curve (see \cite[\S 1.7]{DDT} and the references therein). Proceeding with slightly more analytic sophistication, we recall the definition of the Mellin transform for a complex function $f(x)$ (following as in \cite[\S 5.10]{DS}), with suitable decay conditions at infinity,
\begin{equation}\label{D:Mellin Def}
M(f)(s) = \int_0^\infty f(ix)x^{s-1}dx.
\end{equation}
We then observe that if $f(x)$ is a weight 2 eigenform as mentioned above, then we have
\begin{equation}\label{E:classical modularity}
M(f)(s) = (2\pi)^{-s}\Gamma(s)L_f(s),
\end{equation}
where $\Gamma$ is the traditional gamma function. Thus, the Mellin transform of a cusp form gives a Dirichlet series, and in particular the Mellin transform of a weight 2 newform gives the $L$-function of an elliptic curve.

We wish to study similar questions in the setting of characteristic $p$ valued function fields. In this setting, we have good analogues of each of the objects discussed above: We use Drinfeld modules in place of elliptic curves and Drinfeld modular forms in place of complex valued modular forms. We want to connect Drinfeld modular forms to the $L$-functions attached to Drinfeld modules (called Goss $L$-functions). However, almost immediately from the outset we run into problems. First, we do not have good methods for computing - nor even a good theoretical understanding - of the analogue of the Fourier expansion of Drinfeld modular forms (called a $u$-expansion, see Definition \ref{D:Modular forms def}). Thus the naive connection between Drinfeld modular forms and $L$-functions is missing. Further, due to the characteristic $p$ setting, we lack a meaningful integration theory to compute Mellin or Fourier transforms. Finally, some attempts have been made to connect Galois representations arising from Drinfeld modular forms with those arising from Drinfeld modules (see for example \cite{Gos02} and \cite{Boc}). The main obstruction to carrying out this program seems to be that the eigenspaces of Hecke algebras are always 1-dimensional, so that the local factors arising from them can never match with Drinfeld modules of rank $2$. However, recent results of the second author \cite{Gre24} and of both authors \cite{GG25} have given examples of an algebraic replacement for the Mellin transform which connects special values of $L$-functions and zeta functions with exponential series and rigid analytic trivialization.

In this paper we use the framework of the Mellin transform introduced in \cite{Gre24} to connect special values of certain Drinfeld modular forms  with special values of  Goss $L$-functions. In order to make this connection, we needed to develop new analytic techniques, as compared to \cite{GG25}, which allow us to state the theorems of this paper in much greater generality compared to \cite{GG25}. In particular, the results presented here have sufficient generality so that we can apply them to Drinfeld modules of arbitrary rank which generate vectorial Drinfeld modular forms --- a result which was impossible using the techniques of \cite{GG25}. Our formulas then connect values of these Drinfeld modular forms to values of Goss $L$-functions. We view these theorems as giving tantalizing hints about a deeper modularity connection between Drinfeld modular forms and Goss $L$-functions. We wish to be clear that these theorems are not themselves a modularity theorem --- since they only provide a connection between specified \textit{values} of Drinfeld modular forms and \textit{values} of $L$-functions, rather than establishing a link between the functions themselves. However, this is a very interesting phenomenon and it points in several promising new directions for research, which we will discuss later on in the paper.

\subsection{Precise Description of Results}\label{S:intro1.2}
We begin by introducing some notation and briefly describing the main objects of the present paper. Let $\F_q$ be the finite field with $q:=p^s$ elements ($p$ a prime). Let $A := \F_q[\theta]$ and $K := \F_q(\theta)$. We set $|\cdot|$ to be the non-archimedean norm corresponding to the infinite place normalized so that $|\theta|:=q$. Let $K_\infty := \F_q((1/\theta))$ be the completion of $K$ with respect to $|\cdot|$, and let $\C_\infty$ be the completion of a fixed algebraic closure of $K_\infty$. Let $A\subseteq L \subseteq \mathbb{C}_{\infty}$ be an $A$-algebra. Let $i\in \mathbb{Z}$ and $d,\ell$ be positive integers. For any $B=(b_{\mu \nu})_{\mu \nu}\in \Mat_{d\times \ell}(L)$, we define $B^{(i)}:=(b_{\mu \nu}^{q^{i}})_{\mu \nu}\in \Mat_{d\times \ell}(L)$. Moreover, we define the non-commutative power series ring $\Mat_{d}(L[[\tau]])$ subject to the relation
\[
\tau B=B^{(1)} \tau, \ \ B\in \Mat_{d}(L).
\]
Furthermore, we let $\Mat_d(L[\tau])\subset \Mat_{d}(L[[\tau]])$ be the subring of polynomials in $\tau$ with coefficients in $\Mat_d(L)$. Letting  $t$ be a variable over $\mathbb{C}_{\infty}$, we also define \textit{the Tate algebra $\mathbb{T}$} by 
\[
\mathbb{T}:=\left\{\sum_{i\geq 0}a_it^i\in \mathbb{C}_{\infty}[[t]] \ \ | \ \ |a_i|\to 0 \text{ as } i\to \infty \right\}.
\]
We equip $\mathbb{T}$ with the non-archimedean norm $||\cdot ||$ defined for any $g=\sum_{i\geq 0}a_it^i\in \mathbb{T}$ by 
\[
||g||:=\max\{|a_i| \ \ | \ \ i\in \mathbb{Z}_{\geq 0}\}.
\]
Note that $(\mathbb{T}, ||\cdot ||)$ forms a Banach space. For any $j\in \mathbb{Z}$ and $g=\sum_{i\geq 0}a_it^i\in \mathbb{T}$, we also define 
\[
g^{(j)}:=\sum_{i\geq 0}a_i^{q^j}t^i\in \mathbb{T}.
\]
Furthermore, for any $G=(g_{\mu \nu})_{\mu \nu}\in \Mat_{d\times \ell}(\mathbb{T})$, by a slight abuse of notation, we define
$
||G||:=\max\{||g_{\mu \nu}|| \ \ | \ \ 1\leq \mu, \nu \leq r\}
$. When $G\in \Mat_{d\times \ell}(\mathbb{C}_{\infty})$, we continue to use the earlier notation and let $|G|:=||G||$. Note that $(\Mat_{d\times \ell}(\mathbb{T}), ||\cdot ||)$ forms a Banach space. We further set 
\[
G^{(j)}:=(g_{\mu \nu}^{(j)})_{\mu \nu}\in \Mat_{d\times \ell}(\mathbb{T}).
\]

In \cite{And86}, Anderson defined (abelian) $t$-motives that serve as a function field analogue of motives in the sense of Grothendieck. These are the objects that are free and finitely generated over $\mathbb{C}_{\infty}[\tau]$ and $\mathbb{C}_{\infty}[t]$. We call its rank over $\mathbb{C}_{\infty}[\tau]$ \textit{the dimension of $M$} and its rank over $\mathbb{C}_{\infty}[t]$ \textit{the rank of $M$} (see \S\ref{S:tmot} for more details on $t$-motives). Anderson further showed that there is an anti-equivalence between the category of $t$-motives and the category of \textit{abelian Anderson modules}. By an abelian Anderson module, we mean an $\mathbb{F}_q$-algebra homomorphism $\phi:A\to \Mat_d(\mathbb{C}_{\infty}[\tau])$ which is given by 
       \[
       \phi_{\theta}:=\dmat_{\phi}[\theta]+B_1\tau+\cdots+B_{s}\tau^s
       \]
       for some $s\geq 1$ such that $d_{\phi}[\theta]:=\theta \Id_d+N$ for some nilpotent matrix $N\in \Mat_d(\mathbb{C}_{\infty})$. 
For each $\phi$, there exists a unique infinite series $\Exp_{\phi}:=\sum_{i\geq 0}\alpha_i\tau^i\in \Mat_d(\mathbb{C}_{\infty}[[\tau]])$, which we call \textit{the exponential series of $\phi$}, such that $\alpha_0:=\Id_d$. It induces an everywhere convergent $\mathbb{F}_q$-linear function $\Exp_{\phi}:\mathbb{C}_{\infty}^d\to \mathbb{C}_{\infty}^d$ and we call each non-zero element lying in the kernel $\Ker(\Exp_{\phi})$ \textit{a period of $\phi$}. We call $\phi$  \textit{uniformizable} if $\Exp_{\phi}$ is surjective. We further define \textit{the logarithm series} $\Log_{\phi}$ to be the formal inverse of $\Exp_{\phi}$ in $\Mat_d(\mathbb{C}_{\infty}[[\tau]])$. There exists a subdomain $\mathcal{D}\subset \mathbb{C}_{\infty}^d$ such that the logarithm series induces an $\mathbb{F}_q$-linear function $\Log_{\phi}:\mathcal{D}\to \mathbb{C}_{\infty}^d$ and this function will have a crucial importance in our work. One can also extend $\Log_{\phi}$ to $\Log_{\phi}:\mathcal{D}'\to \mathbb{T}^d$ where $\mathcal{D}'$ is a subset of $\mathbb{T}^d$ containing $\mathcal{D}$ (see \eqref{E:extension}). For more details on these aforementioned objects, we refer the reader to \S\ref{S:Andersondmod}.

Let $M$ be a $t$-motive of rank $r$ and dimension $d$. We let $\bm\in \Mat_{r\times 1}(M)$ be a $\mathbb{C}_{\infty}[t]$-basis for $M$ and consider the matrix $\Phi\in \Mat_{r}(\mathbb{C}_{\infty}[t])$ so that $\Phi^{\tr}$ represents the $\tau$-action on $\bm$. On the other hand, we let $\bb\in \Mat_{d\times 1}(M)$ be a $\mathbb{C}_{\infty}[\tau]$-basis for $M$ and consider the matrix $\mathfrak{M}\in \Mat_d(\mathbb{C}_{\infty}[\tau])$ that represents the $t$-action on $\bb$. We define the abelian Anderson module $\phi$ corresponding to $M$  by $\phi_{\theta}:=\mathfrak{M}$ and further assume that it is uniformizable. In this case, by \cite[Thm. 4]{And86}, we know that $\Ker(\Exp_{\phi})$ is a free $A$-module of rank $r$, and we fix an $A$-module basis  $\{\lambda_1,\dots,\lambda_r\}$ for $\Ker(\Exp_{\phi})$.  Using these aforementioned bases, in \eqref{E:isomtmot}, we construct a $\mathbb{C}_{\infty}[t,\tau]$-module isomorphism between $M\cong \Mat_{1\times d}(\mathbb{C}_{\infty}[\tau])$ and $\Mat_{1\times r}(\mathbb{C}_{\infty}[t])$. We further consider the matrix $\Psi\in \GL_r(\mathbb{T})$ satisfying
\[
\Psi^{(-1)}=\Phi \Psi.
\]
We note that $\Psi$ is indeed \textit{a rigid analytic trivialization of the $t$-comotive (also known as the dual $t$-motive in the literature) of $\phi$}. We refer the reader to \S\ref{S:uniAndmod} and \S\ref{S:dualmot} for further details.

For each $1\leq \nu \leq r$, we define the Anderson generating function of $\phi$ with respect to $\lambda_{\nu}$ by
\[
\mathcal{Y}_{\lambda_v}(t):=\sum_{i=0}^{\infty}\Exp_{\phi}\left(\dmat_{\phi}[\theta]^{-i-1}\lambda_{\nu}\right)t^i\in \mathbb{T}^d
\] and refer the reader to \S\ref{S:Andgen} for further details on these objects.
In \S\ref{S:varphimap}, using Anderson generating functions, we construct a certain subset $\mathbb{M}_{\phi}$ of $\Mat_{1\times d}(\mathbb{C}_{\infty}[[\tau]])$ which is equipped with a normed $\mathbb{C}_{\infty}$-vector space structure and contains $M$ as a dense subset.  Furthermore, we define an injective  continuous map $\varphi_{\phi}:\mathbb{M}_{\phi} \to \Mat_{1\times r}(\mathbb{T})$ that extends the isomorphism defined in \eqref{E:isomtmot} between $M$ and $\Mat_{1\times r}(\mathbb{C}_{\infty}[t])$ (Theorem \ref{T:tenwelinj}). Additionally, following the work of the second author \cite{Gre24}, for any $\zz\in \mathbb{C}_{\infty}^d$, we consider the map $\delta_{1,\bz}^{M}:\mathbb{M}_{\phi}\to \mathbb{C}_{\infty}^d$ which recovers the structure of the abelian Anderson module $\phi$ (see \S\ref{S:varphimap} for a precise definition). Lastly, we define \begin{equation}\label{E:mzmap}
\mathcal{M}_{\zz}:=\delta_{1,\zz}^{M}\circ \varphi_{\phi}^{-1}.
\end{equation}

We view $\mathcal{M}_{\zz}$ as an algebraic version of the Mellin transform in our framework. We give a short survey on the work related to this Mellin transform and explain our viewpoint. In \cite{Gre24}, the second author proved a formula relating the Mellin transform of the exponential function of the $n$th tensor power of the Carlitz module $C^{\otimes n}$ (see \S \ref{S:Carlitztensor}) to \textit{the $n$-th Carlitz zeta value $\zeta_{A}(n)$} given by 
\[
\zeta_{A}(n):=\sum_{\substack{a\in A\\a\text{ is monic}}}\frac{1}{a^n}.
\]
Specifically, he showed in \cite[Cor. 5.11]{Gre24} that for $u = \tpi^n/(\theta-t)$ and for specified $\zz \in \C_\infty^n$ we have
\[\mathcal{M}_{\zz}\left (\frac{u}{\mathfrak{p}_{n}(\Exp_{C^{\otimes n}}(u))}\right ) = \Gamma_A(n)\zeta_A(n),\]
where $\Gamma_A$ and $\zeta_A$ are function field versions of the Gamma and the zeta function, respectively and $\mathfrak{p}_{n}$ is the projection onto the $n$th coordinate. This formula should be compared with the classical formula using the complex valued Mellin transform described above which gives the completed Riemann zeta function:
\[M\left(\frac{1}{e^x-1}\right )(s) = \Gamma(s)\zeta(s).
\]

In \cite{GG25}, the authors extended the techniques and themes of \cite{Gre24} to apply to a wider class of Anderson modules, rather than just tensor powers of the Carlitz module. They proved formulas relating the Mellin transform of rigid analytic trivializations (see \S\ref{S:Drinfeld}) to special values of Goss $L$-functions for tensor powers of certain Drinfeld modules tensored with powers of the Carlitz module. In fact, the theorems of \cite{GG25} can be seen as a special, restricted case of Theorem \ref{T:Intro log formula} presented here. In order for the analytic techniques of \cite{GG25} to work, the authors required strong conditions on the Drinfeld modules studied. The present paper should be considered as a natural continuation of the arc of these ideas. In particular, we develop new analytic techniques (explained in Remark \ref{R:Analytic techniques}) which allow us to apply these techniques for a very large class of Anderson modules. We simply require it to be abelian and uniformizable. This generality allows us to give the main applications of this paper: The formula relating values of Drinfeld modular forms with values of dual Goss $L$-functions (Theorem \ref{T:Intromodform}).

We now proceed to describe the precise results of this paper. By Lemma \ref{L:integerm}, there exists a unique smallest non-negative integer $\mathfrak{m}_{\phi}$ such that for each $\nu$ and $i\geq 1$, the element 
$
\theta^{-\mathfrak{m}_{\phi}}\Exp_{\phi}\left(\dmat_{\phi}[\theta]^{-i}\lambda_{\nu}\right)\in \mathbb{C}_{\infty}^d$
lies in the domain of convergence of $\Log_{\phi}$. This implies that for each $\nu$, the element  $\theta^{-\mathfrak{m}_{\phi}}\mathcal{Y}_{\lambda_{\nu}}(t)\in \mathbb{T}^d$ lies in the extended domain $\mathcal{D}'$ of the logarithm function $\Log_{\phi}$ of $\phi$. We write
\[
\Log_{\phi}=\sum_{n=0}^{\infty}L_n\tau^n\in \Mat_d(\mathbb{C}_{\infty}[[\tau]]),\ \ \ \ L_n=(L_{ij}^{[n]})_{i,j}\in \Mat_{d}(\mathbb{C}_{\infty}).
\]
Next, for $1\leq \ell \leq d$, we set 
\begin{equation}\label{D:mathcal P}
\mathcal{P}_{\ell}(\mathcal{Z}):=\sum_{n\geq 0}\frac{1}{\theta^{q^n\mathfrak{m}_{\phi}}}[L^{[n]}_{\ell 1},\dots,L^{[n]}_{\ell d}]\begin{pmatrix}
    \mathcal{Z}_1^{(n)}\\
    \vdots\\
    \mathcal{Z}_d^{(n)}
\end{pmatrix},  \ \ \mathcal{Z}=(\mathcal{Z}_1,\dots,\mathcal{Z}_d)^{\tr}\in \mathbb{T}^d.
\end{equation}
By our condition on $\mathfrak{m}_{\phi}$, note that each $\mathcal{P}_{\ell}(\mathcal{Y}_{\nu}(t))$ converges in $\mathbb{T}$.

The first main result of the present paper, generalizing \cite[Thm. 1.1, Thm. 1.6]{GG25} to abelian and uniformizable Anderson modules, can be stated as follows.
\begin{theorem}\label{T:Intro log formula} Let $\mathfrak{p}_{\ell}:\mathbb{C}_{\infty}^{d}\to \mathbb{C}_{\infty}$ be the projection onto the $\ell$-th coordinate. For each $1\leq \nu \leq r$, we also let $\lambda_{\nu}=(\lambda_{\nu 1},\dots,\lambda_{\nu d})\in \mathbb{C}_{\infty}^{d}$. Then, for any $\zz\in \mathbb{C}_{\infty}^d$ lying in the domain of convergence of $\Log_{\phi}\theta^{-\mathfrak{m}_{\phi}}$, we have
\begin{equation}\label{E:formula}
\mathfrak{p}_{\ell}(\Log_{\phi}(\theta^{-\mathfrak{m}_{\phi}}\zz))=
\mathcal{M}_{\zz}\left([\mathcal{P}_{\ell}(\mathcal{Y}_1(t)),\dots, \mathcal{P}_{\ell}(\mathcal{Y}_r(t))](\Psi^{\tr})^{(-1)}\right).
\end{equation}
Moreover, let $1\leq \mathfrak{t}\leq d$ be a tractable coordinate of $\phi$ (see \S \ref{S:Andersondmod} for the definition of a tractable coordinate of $\phi$). If $\mathfrak{m}_{\phi}=0$ and for $i\geq 1$, each $\dmat_{\phi}[\theta]^{-i}\lambda_{\nu}$ lies in the domain of convergence of $\Log_{\phi}$, then we have 
\[
\mathfrak{p}_{\mathfrak{t}}(\Log_{\phi}(\zz))=\mathcal{M}_{\zz}\left(\frac{1}{\theta-t}[\lambda_{1 \mathfrak{t}},\dots, \lambda_{r \mathfrak{t}}](\Psi^{\tr})^{(-1)}\right).
\]

\end{theorem}

\begin{remark}\label{R:Analytic techniques} In \cite[Thm. 1.1, Thm. 1.6]{GG25}, the authors proved Theorem \ref{T:Intro log formula} for a certain class of Drinfeld modules and for tensor products of such Drinfeld modules with Carlitz tensor powers, which are also examples of abelian and uniformizable Anderson modules. There, we obtained our results via a delicate analysis which is mainly based on the coefficients of Drinfeld modules. In particular, we required the coefficients of the Drinfeld module to have norm less than or equal to 1. We also emphasize that the integer $\mathfrak{m}_{\phi}$ corresponding to Anderson modules studied in \cite{GG25} is zero. In the present paper, we instead focus on an analysis of Anderson generating functions. This approach not only provides us a more general and complete picture but also allows us to avoid complicated computations heavily depending on the norm of the coefficients of Anderson modules.
\end{remark}

\subsection{Special values of $L$-functions}\label{S:Intro1.3}
Our next result concerns the special values of  $L$-functions of Drinfeld modules. First, we briefly describe dual Goss $L$-functions attached to Drinfeld modules introduced by Goss \cite{Gos92} whose constructions are inspired by the ideas of Gekeler \cite[Rem. 5.10]{Gek91}. Let $\phi$ be a Drinfeld module defined over a finite extension of $K$ (see \S\ref{S:Drinfeld} for a discussion on Drinfeld modules). Let $\mathfrak{v}\in A$ be a monic irreducible polynomial. We set $K_{\mathfrak{v}}$ to be the completion of $K$ at the place corresponding to $\mathfrak{v}$. Let $(\rho_{\mathfrak{v}})$ be a family of continuous representations of the Galois group of $K^\text{sep}/K$ 
so that the characteristic polynomial 
\[P_v(X) := \det(1-X\cdot \rho_{\mathfrak{v}}(\Frob_v))\]
of the Frobenius map at a place $v\neq {\mathfrak{v}}$ of $K$ acting on the $\mathfrak{v}$-adic Tate module of $\phi$ is independent of the choice of prime $\mathfrak{v}$. Moreover, it has coefficients in $A$ (along with a ramification condition - see \cite[\S8.10]{Gos96} for full details). We also let $P_v(X)=(1-a_1X)\cdots (1-a_rX)$ for some $a_1,\dots,a_r$ in a fixed algebraic closure of $K$ in $\mathbb{C}_{\infty}$ and set 
\[
P_v^{\vee}(X):=(1-a_1^{-1}X)\cdots (1-a_r^{-1}X).
\]
Then \textit{the  dual Goss $L$-function of $\phi$} is given by
\begin{equation}\label{E:L func Euler expansion}
L(\phi^{\vee},n) := \prod_{v} P_v^{\vee}(v^{-n})\inv.
\end{equation}
Here, the product runs over all the finite places of $A$. By \cite[Cor. 3.6]{ChangEl-GuindyPapanikolas}, we know that $L(\phi^{\vee},n)$ converges in $K_{\infty}$ for all $n\in \mathbb{Z}_{\geq 0}$. We also note that if $\phi$ is \textit{the Carlitz module} given by $C_{\theta}:=\theta+\tau$, then  for any positive integer $n$, we have $L(C^{\vee},n-1)=\zeta_{A}(n)$.

We will now describe our next result. We first consider the Drinfeld upper half plane
\[
\mathbb{H}^r=\mathbb{P}^{r-1}(\CC_{\infty})\setminus \{K_{\infty}\text{-rational hyperplanes}\}.
\]
We identify any of its elements as $\ww=(w_1,\dots,w_r)^{\tr}\in \CC_{\infty}^{r}$ whose entries are $K_{\infty}$-linearly independent and normalized so that $w_r=1$. For any $\gamma=(a_{ij})\in \GL_r(K_\infty)$, we define the action of $\GL_r(K_{\infty})$ on $\mathbb{H}^r$ by
\[
\gamma \cdot \ww:=\Big(\frac{a_{11}w_1+\dots+ a_{1r}w_r}{a_{r1}w_1+\dots+ a_{rr}w_r},\dots,\frac{a_{(r-1)1}w_1+\dots +a_{(r-1)r}w_r}{a_{r1}w_1+\dots+ a_{rr}w_r},1\Big)^{\tr}\in \mathbb{H}^r.
\]
By Drinfeld \cite{Dri74}, we know that the moduli space of isomorphism classes of Drinfeld modules of rank $r$ defined over $\mathbb{C}_{\infty}$ is the quotient $\GL_r(A)\setminus \HH^r$. Thus, up to passing to an isomorphic Drinfeld module if necessary, we assume that the period lattice $\Lambda_{\ww}:=Aw_1+\cdots+Aw_r$  generated by the entries of $\ww$, under the one-to-one correspondence described by Drinfeld \cite[\S5]{Dri74}, produces a Drinfeld module $\phi(\ww)$ defined by 
\begin{equation}\label{E:Drinfeldz}
    (\phi(\ww))_{\theta} := \theta + g_1(\ww)\tau + \cdots+g_r(\ww)\tau^r
\end{equation}
so that $g_1(\ww),\dots,g_r(\ww)$ are in $A$. For each $1\leq i \leq r$, we further let $f_i(\ww,t)\in \mathbb{T}$ be the  Anderson generating function of $\phi(\ww)$ with respect to $\ww_i$.

For any $\mathfrak{p}\in A$, we let $\mathfrak{p}(t):=p_{|\theta=t}\in \mathbb{F}_q[t]$ and 
for $\gamma = (a_{ij})_{i,j} \in \GL_r(A)$, define
\[\overline \gamma: = (a_{ij}(t))_{i,j} \in \GL_r(\mathbb{F}_q[t]).\]
We further set
\[
j(\gamma,\ww):=a_{r1}w_1+\dots+ a_{rr}w_r\in \CC_{\infty}^{\times}.
\]
The map $\gamma \mapsto \overline \gamma$ can be viewed as an example of a representation of $\GL_r(A)$ on a Banach space described in \cite[(1.3)]{Pel25}. We call this \textit{the identity representation}.
Inspired by the work of Pellarin \cite{Pel12} and Pellarin and Perkins \cite{PP18}, we further define another representation $\rho^{*}$ of $\GL_r(A)$ given by the map $\gamma \to (\overline \gamma^{\tr})^{-1}$. 

A rigid analytic matrix-valued function $f:\mathbb{H}^r \to \Mat_{r}(\TT)$, equipped with a certain growth condition (see \eqref{D:matrix modular type}), is \textit{a modular matrix-valued function with weights $(\ell_1,\dots,\ell_r) \in \mathbb{Z}^r$ and type $m\in \Z/(q-1)\Z$ with respect to $\rho^{*}$}
 if for all $\gamma\in \GL_r(A)$, we have
\[f\left(\gamma \cdot \ww\right) =\det(\gamma)^{-m}\rho^{*}(\gamma) f(\ww) \begin{pmatrix}
    j(\gamma,\ww)^{\ell_1}& & & \\
     & \ddots & &  \\
      & & & j(\gamma,\ww)^{\ell_r}
      \end{pmatrix}.\]
If a function $f:\mathbb{H}^r \to \Mat_{r}(\TT)$ is such that $f\twistk{k}$ is modular  matrix-valued function for $\rho^{*}$ where $k\in \Z_+$ with $k$ minimal with weights $(q^k\ell_1,\dots,q^k\ell_r)$ and type $m$, then we say that $f$ is \textit{a modular  matrix-valued function for $\rho^{*}$ with weights $(\ell_1,\dots,\ell_r)\in \QQ^r$, type $m$ and root $q^k$}. These rigid analytic functions should be compared to classical vector-valued modular forms.

Our second main result concerns a certain relation between modular matrix valued functions and our Mellin transform, which will be restated as Proposition \ref{P:modlike} and Corollary \ref{C:modlike} later. 
\begin{theorem}\label{T:Intromodform}
Let $\ww=(w_1,\dots,w_r)^{\tr}\in \mathbb{H}^r$ and   $\phi(\ww)$ be as in \eqref{E:Drinfeldz}. Let $\Psi(\ww)\in \GL_r(\TT)$ be the rigid analytic trivialization associated to $\phi(\ww)$ as in \eqref{E:ratmod}.
\begin{itemize}
\item[(i)] The function $\ww\to (\Psi(\ww)^{\tr})^{(-1)}$ is a modular matrix-valued function of weights $(1/q,\dots,1/q^{r-1})$,  type $0$ and root $q^r$ with respect to $\rho^{*}$. 
\item[(ii)] The Mellin transform of $\Psi(\ww)$, up to a constant $\mathfrak{c}_{\phi(\ww)}\in K^{\times}$, equals the dual Goss $L$-function associated to $\phi(\ww)$ evaluated at $0$. More precisely, there exist an element $\mathfrak{f}_{\phi(\ww)}\in K_{\infty}^{\times}$ and a non-negative integer $N(\phi(\ww))$ so that if we let $\bomega:= \theta^{\mathfrak{m}_{\phi(\ww)}}\Exp_{\phi(\ww)}(\mathfrak{f}_{\phi(\ww)}/\theta^{N(\phi(\ww))})$ and set 
\[
\mathcal{I}(\phi(\ww)):=\left[\Log_{\phi(\ww)}\left(\frac{f_{1}(\ww,t)}{\theta^{\mathfrak{m}_{\phi(\ww)}}}\right),\dots,\Log_{\phi(\ww)}\left(\frac{f_{r}(\ww,t)}{\theta^{\mathfrak{m}_{\phi(\ww)}}}\right)\right]\in \Mat_{1\times r}(\mathbb{C}_{\infty}),
\]
then we have
\begin{equation}\label{E:ourformula}
L(\phi(\ww)^{\vee},0)=\mathfrak{c}_{\phi(\ww)}\mathcal{M}_{\bomega}\left(\mathcal{I}(\phi(\ww))
(\Psi(\ww)^{\tr})^{(-1)}\right).
\end{equation}
\item[(iii)] Let $\phi(\ww)$ be the Drinfeld module of rank two as in \eqref{E:Drinfeldz} so that for $i=1,2$, we have $\deg_{\theta}g_i(\ww)<q^{i}$. Then $\bomega=1$ and 
\[
L(\phi(\ww)^{\vee},0)=\mathfrak{c}_{\phi(\ww)}\mathcal{M}_{\bomega}\left(\frac{1}{\theta-t}
\left[w_1,w_2\right]
(\Psi(\ww)^{\tr})^{(-1)}\right).
\]
\end{itemize}

\end{theorem}

We view the above theorem as giving a tantalizing hint about modularity for Drinfeld modular forms. As discussed in the previous section, the exact nature of the relationship between Drinfeld modular forms and $L$-functions of Drinfeld modules remains quite mysterious. The above theorem is the first example (to the authors' knowledge) of establishing a direct connection between a Drinfeld modular form and special values of Drinfeld $L$-functions. We again emphasize that the above theorem is not a true modularity theorem, because it relates \textit{values} of Drinfeld modular forms with \textit{values} of Goss $L$-functions. Nevertheless, our formula \eqref{E:ourformula} should be compared to the classical formula discussed in \eqref{E:classical modularity}.

Further, if one changes the value of $\ww$ in the above theorem, it changes the \textit{value} of the Drinfeld modular form on the right hand side, whereas it changes the Drinfeld module, and consequently, the \textit{Goss $L$-function itself}, on the left hand side. Ideally, one would wish for $\ww$ to appear in the argument of the $L$-function on the left hand side, rather than in the Drinfeld module. The exact nature of how this formula depends on choice of $\ww$ is subtle, and is a topic of future work.

\begin{remark} Using Theorem \ref{T:Intro log formula}, we can also obtain closed formulas for the special values of Goss $L$-functions at any positive integer $k$ in terms of our Mellin transform. This requires connecting the Anderson module $\phi\otimes C^{\otimes k}$ for any $k\geq 1$ to special values of $L$-functions of Drinfeld modules defined over $A$ and analysis on periods of $\phi\otimes C^{\otimes k}$. For the sake of reducing the complexity of the present paper, we reserve our upcoming paper \cite{GG26} to the aforementioned study of special values.
\end{remark}

The outline of the present paper can be described as follows. In \S2, we introduce Anderson modules as well as the objects attached to Anderson modules such as Anderson generating functions, $t$-motives and $t$-comotives. We also discuss the case of Drinfeld modules and the tensor powers of the Carlitz module in \S\ref{S:Drinfeld} and \S\ref{S:Carlitztensor}. In \S3, we provide a proof for Theorem \ref{T:Intro log formula}. In \S4, we discuss details of modular matrix valued functions for $\rho^{*}$ and prove Theorem \ref{T:Intromodform}.

\subsection*{Acknowledgments} The authors would like to thank Yen-Tsung Chen for his comments on an earlier version of the paper which later inspire us to generalize Theorem \ref{T:Intromodform} for Drinfeld modules defined over $A$. The first author acknowledges support by NSTC Grant 113-2115-M-007-001-MY3. The second author acknowledges support from the state of Louisiana Board of Regents and from the NSF under Grant No. 2302399.

\section{Preliminaries and Background}
In this section, we briefly describe the main tools of the present paper, namely, Anderson modules, Anderson generating functions, (abelian) $t$-motives as well as ($t$-finite) $t$-comotives. In the last subsection, we also provide several examples of these objects. Our exposition is mainly based on \cite{And86, BP20, HJ20, NP21} and hence we refer the reader to those references for further details.
\subsection{Anderson modules} \label{S:Andersondmod} For any $A$-algebra $L$, we start by recalling the non-commutative polynomial power series ring $ \Mat_{d}(L[[\tau]]) $ and the polynomial ring $ \Mat_{d}(L[\tau]) $ from \S\ref{S:intro1.2}. Each element $\mathcal{U}=\sum_{i\geq 0}B_i\tau^i\in \Mat_{d}(L[[\tau]])$ defines an action on $L^d:=\Mat_{d\times 1}(L)$ given by
\[
\mathcal{U}\cdot x=\sum_{i=0}^{\infty}B_ix^{(i)}, \ \ x\in L^d
\]
provided that the right hand side of the above equality converges in $L$.

\begin{definition}
   \begin{itemize}
       \item[(i)] \textit{An Anderson module of dimension $d$} (defined over $L$) is an $\mathbb{F}_q$-algebra homomorphism $\phi:A\to \Mat_d(L[\tau])$ which is given by 
       \[
       \phi_{\theta}=\dmat_{\phi}[\theta]+B_1\tau+\cdots+B_{s}\tau^s
       \]
       for some $s\geq 1$ and $d_{\phi}[\theta]=\theta \Id_d+N$ for some nilpotent matrix $N\in \Mat_d(L)$.
       \item[(ii)] Let $\psi$ be an Anderson module of dimension $d'$.  A morphism between $\phi$ and $\psi$ is given by a matrix $U\in \Mat_{d\times d'}(L[\tau])$ satisfying
       \[
       \phi_{\theta}U=U\psi_{\theta}.
       \]
       If such a $U$ exists, then we say that $\phi$ and $\psi$  are \textit{isogenous} (over $L$). Furthermore, if $U\in \GL_d(L[\tau])$, then  we say that $\phi$ and $\psi$ are \textit{isomorphic} (over $L$).
   \end{itemize} 
\end{definition}
We denote by $\Lie(\phi)(L)$ the $A$-module $L^d$ whose $A$-module structure is given by
\[
a\cdot x:=\dmat_{\phi}[a]x, \ \ x\in L^d.
\]For an Anderson module $\phi$ of dimension $d$ and $1\leq i \leq d$, we say that its $i$-th coordinate is \textit{tractable} if, for any $z\in \mathbb{C}_{\infty}^d$, the $i$-th coordinate of $a\cdot z\in \Lie(\phi)(\mathbb{C}_{\infty})$ is equal to $az$ for all $a\in A$.
Further, we denote by $\phi(L)$ the $A$-module $L^d$ whose $A$-module structure is given by
\[
a\cdot x:=\phi_a x, \ \ a\in A, \ \  \  x\in L^d.
\]
Let $\Frac(L)$ be the fraction field of $L$. There exists a unique infinite series $\Exp_{\phi}=\sum_{i\geq 0}\alpha_i\tau^i\in \Mat_d(\Frac(L)[[\tau]])$, which we call \textit{the exponential series of $\phi$}, such that $\alpha_0=\Id_d$ and it satisfies the following functional equation in $\Mat_d(\Frac(L)[[\tau]])$:
\[
\Exp_{\phi}\dmat_{\phi}[\theta]=\phi_{\theta}\Exp_{\phi}.
\]
It further induces an everywhere convergent $\mathbb{F}_q$-linear function $\Exp_{\phi}:\mathbb{C}_{\infty}^d\to \mathbb{C}_{\infty}^d$ given by
\[
\Exp_{\phi}(\zz)=\sum_{i\geq 0}\alpha_i\zz^{(i)},\ \ \zz\in \Lie(\phi)(\mathbb{C}_{\infty}).
\]

By \cite[\S2.2]{And86}, we know that $\Exp_{\phi}$ may not be always surjective. If $\Exp_{\phi}$ is surjective, then we say that $\phi$ is \textit{uniformizable}. We set $\Lambda_{\phi}:=\Ker(\Exp_{\phi})\subset \Lie(\phi)(\mathbb{C}_{\infty})$ and call $\Lambda_{\phi}$ \textit{the period lattice of $\phi$}. It is indeed a free, finitely generated and discrete $A$-module. We further call its non-zero elements  \textit{periods of $\phi$}. 

\textit{The logarithm series $\Log_{\phi}=\sum_{i\geq 0} \gamma_i \tau^i\in \Mat_d(\Frac(L)[[\tau]])$ of $\phi$} is defined to be the formal inverse of $\Exp_{\phi}$. It satisfies the following functional equation in $\Mat_d(\Frac(L)[[\tau]])$:
\[
\Log_{\phi}\phi_{\theta}=\dmat_{\phi}[\theta]\Log_{\phi}.
\]
Moreover, there exists a domain $\mathcal{D}\subset \mathbb{C}_{\infty}^d$ such that the logarithm series induces an $\mathbb{F}_q$-linear function $\Log_{\phi}:\mathcal{D}\to \mathbb{C}_{\infty}^d$ which is given by
\[
\Log_{\phi}(\zz)=\sum_{i\geq 0}\beta_i\zz^{(i)},\ \ \zz\in \mathcal{D}\subset \phi(\mathbb{C}_{\infty}).
\]
For later use, we extend the domain of the exponential function of $\phi$ and define $\Exp_{\phi}: \mathbb{T}^d\to \mathbb{T}^{d}$ so that 
\[
\Exp_{\phi}(f)=\sum_{i\geq 0}\alpha_if^{(i)},\ \ f\in \mathbb{T}.
\]
Similarly, we also extend the domain of the logarithm function of $\phi$ and consider $\Log_{\phi}: \mathcal{D}'\to \mathbb{T}^{d}$ so that 
\begin{equation}\label{E:extension}
\Log_{\phi}(f)=\sum_{i\geq 0}\beta_if^{(i)},\ \ f\in \mathcal{D}'
\end{equation}
where $\mathfrak{D}'$ is a subdomain of $\mathbb{T}^d$ containing $\mathcal{D}$.

\subsection{Anderson generating functions}\label{S:Andgen} 
Let $\phi:A\to \Mat_d(L[\tau])$ be an Anderson module and $\zz\in \mathbb{C}_{\infty}^d$. We define \textit{the Anderson generating function $\mathcal{Y}_{\zz}(t)$ of $\phi$ with respect to $\zz$} to be the infinite sum given by 
\[
\mathcal{Y}_{\zz}(t):=\sum_{i=0}^{\infty}\Exp_{\phi}\left(\dmat_{\phi}[\theta]^{-i-1}\zz\right)t^i\in \mathbb{C}_{\infty}[[t]]^d.
\]
Let $B=B_0+B_1\tau+\dots+B_u\tau^u\in \Mat_{m\times \ell}(\mathbb{C}_{\infty}[\tau])$ and $\mathcal{G}\in \Mat_{\ell \times v}(\mathbb{T})$. Following, \cite[(2.1.2)]{NP21}, we set 
\[
\langle B|\mathcal{G}\rangle:= B_0\mathcal{G}+B_1\mathcal{G}^{(1)}+\cdots+B_u\mathcal{G}^{(u)}\in \Mat_{m\times v}(\mathbb{T}).
\]
We now list several fundamental properties of Anderson generating functions.
\begin{lemma}\cite[Prop. 4.2.7, Lem. 4.2.2, Prop. 4.2.12]{NP21} \label{L:AGF} The following statements hold.
\begin{itemize}
    \item[(i)] We have the identity
    \[
    \mathcal{Y}_{\zz}(t)=\sum_{i=0}^{\infty}\alpha_i\left((\dmat_{\phi}[\theta]-t\Id_d)^{-1}\right)^{(i)}\zz^{(i)}
    \]
    which converges in $\mathbb{T}^d=\Mat_{d\times 1}(\mathbb{T})$.
    \item[(ii)] Let $\mathcal{Y}_{\zz}(t)=((\mathcal{Y}_{\zz})_1,\dots,(\mathcal{Y}_{\zz})_d)^{\tr}\in \mathbb{T}^d.$ Then
    \[
    \Res_{t=\theta}\mathcal{Y}_{\zz}:=(\Res_{t=\theta}(\mathcal{Y}_{\zz})_1,\dots, \Res_{t=\theta}(\mathcal{Y}_{\zz})_d)^{\tr}=-\zz.
    \]
    \item[(iii)] Let $\lambda \in \Lambda_{\phi}$. Then
    \[
    \langle\phi_{\theta}|\mathcal{Y}_{\lambda}(t)\rangle=t\mathcal{Y}_{\lambda}(t).
    \]
\end{itemize}
    
\end{lemma}
\subsection{$t$-motives} \label{S:tmot} Throughout this subsection, we assume that $ L\subseteq \mathbb{C}_{\infty}$ is a perfect field containing $K$. We define the ring $L[t,\tau]$ of polynomials in $t$ and $\tau$ with coefficients in $L$ subject to the relations
\[
tx=xt, \ \ t\tau=\tau t, \ \ \tau x=x^q \tau, \ \ x\in L.
\]

\begin{definition}
	\begin{itemize}
		\item[(i)]  \textit{A (abelian) $t$-motive $M$ over $L$} is a left $L[t,\tau]$-module which is free and finitely generated over $\mathbb{C}_{\infty}[t]$ and $\mathbb{C}_{\infty}[\tau]$ so that  there exists a non-negative integer $\nu$ satisfying
		\[
		(t-\theta)^{\nu}M\subset \tau M.
		\] 
		\item[(ii)] Let $M_1$ and $M_2$ be $t$-motives over $L$.  Morphisms between $M_1$ and $M_2$ are given by left $L[t,\sigma]$-module homomorphisms.
		\item[(iii)]  \textit{The tensor product of $M_1$ and $M_2$} forms a $t$-motive $M_1\otimes_{L[t]}M_2$ where $\tau$ acts diagonally.
	\end{itemize}
\end{definition}

Let $\{\widetilde{\bm}_1,\dots,\widetilde{\bm}_r\}$ be an $L[t]$-basis for $M$ and let $\Theta\in  \Mat_r(L[t])$ be such that 
\begin{equation}\label{E:003}
\tau \cdot \begin{bmatrix}\widetilde{\bm}_1\\ \vdots\\
\widetilde{\bm}_r
\end{bmatrix}=\Theta\begin{bmatrix}\widetilde{\bm}_1\\ \vdots\\
\widetilde{\bm}_r
\end{bmatrix}.
\end{equation}
A $t$-motive $M$ is called \textit{rigid analytically trivial} if there exists $\Upsilon\in \GL_r(\mathbb{T})$ such that 
\[
\Upsilon^{(1)}=\Theta\Upsilon.
\]
We  call $\Upsilon$ \textit{a rigid analytic trivialization of $M$}.

Due to Anderson, there exists an anti-equivalence between a subcategory of Anderson modules, namely the category of \textit{abelian Anderson modules defined over $L$} and the category of $t$-motives over $L$. In what follows, we briefly explain this correspondence. 

Let $\phi$ be an Anderson module of dimension $d$ which is defined over $L$. We set $M_{\phi}:=\Mat_{1\times d}(L[\tau])$. Note that $M_{\phi}$ is naturally a free and finitely generated $L[\tau]$-module.  Moreover it is equipped with an $L[t]$-module structure given by 
\begin{equation}\label{E:tmodstr}
\alpha t^i\cdot m:=\alpha  m  \phi_{\theta^i}, \ \ m\in M_\phi, \ \ \alpha\in L.
\end{equation}
Hence, $M_{\phi}$ forms a left $L[t,\tau]$-module. If $M_{\phi}$ is also a free $L[t]$-module of finite rank, we say that $\phi$ is \textit{an abelian Anderson module}.

On the other hand, let $\mathbf{b}:=\{\bb_1,\dots,\bb_d\}$ be an $L[\tau]$-basis for a $t$-motive $M$. Then there exists a matrix $\mathfrak{M}\in \Mat_d(L[\tau])$ such that 
\begin{equation}\label{E:002}
t\cdot \mathbf{b}=\mathfrak{M}\mathbf{b}.
\end{equation}
Since $t$ commutes with the elements in $L[\tau]$, representing an arbitrary element of $M$ as
\[
(c_1,\dots,c_d)\begin{bmatrix}
    \bb_1\\
    \vdots\\
    \bb_d
\end{bmatrix}=\mathfrak{c}\mathbf{b}
\]
where $c_1,\dots,c_d\in L[\tau]$ and $\mathfrak{c}:=[c_1,\dots,c_d]\in \Mat_{1\times d}(L[\tau])$, we then obtain 
\[
\alpha t^i \cdot \mathfrak{c}\mathbf{b}=\alpha \mathfrak{c}\cdot t^i \mathbf{b}=\alpha \mathfrak{c}\mathfrak{M}^i\mathbf{b}, \ \ \alpha \in L.
\]

Now setting $\phi:A\to \Mat_{d}(L[\tau])$ to be the $\mathbb{F}_q$-algebra homomorphism given by $\phi_{\theta}:=\mathfrak{M}$ and using the fact that $(t-\theta)^{\nu}M\subset \tau M$ for some $\nu \geq 0$, we see that $\phi$ forms an Anderson module. Furthermore, note that the above construction is indeed compatible with the $L[t]$-module action described in \eqref{E:tmodstr}. Hence, $\phi$ is abelian.

\subsection{$t$-comotives} \label{S:dualmot} Our goal is to review the notion of $t$-comotives. We continue to assume that $L\subseteq \mathbb{C}_{\infty}$ is a perfect field and define the non-commutative polynomial ring $\Mat_{d}(L[\sigma])$ subject to the relation
\[
\sigma B=B^{(-1)} \sigma, \ \ B\in \Mat_{d}(L).
\]
We define the ring $L[t,\sigma]$ of polynomials in $t$ and $\sigma$ with coefficients in $L$ subject to the relations
\[
tx=xt, \ \ t\sigma=\sigma t, \ \ \sigma x=x^{(-1)}\sigma=x^{1/q} \sigma, \ \ x\in L.
\]
We further define the $*$-operator acting on $ g= \sum_{i\geq 0}c_i\tau^i\in L[\tau]$  given by 
\[
g^{*}:=\sum_{i\geq 0}c_i^{(-i)}\sigma^i\in L[\sigma].
\]
For any $B=(b_{ij})_{ij}\in \Mat_d(L[\tau])$, we also let
\[
B^{*}:=(b_{ij}^{*})^{\tr}_{ij}\in \Mat_d(L[\sigma]).
\]
\begin{definition}
	\begin{itemize}
		\item[(i)] \textit{A ($t$-finite) $t$-comotive $N$} is a left $L[t,\sigma]$-module which is free and finitely generated over $L[t]$ and $L[\sigma]$ so that there exists $\mu\in \mathbb{Z}_{\geq 0}$ satisfying
		\[
		(t-\theta)^{\mu} N\subset \sigma N.
		\] 
		\item[(ii)] Let $N_1$ and $N_2$ be $t$-comotives. Morphisms between $N_1$ and $N_2$ are given by left $L[t,\sigma]$-module homomorphisms.
		\item[(iii)] The tensor product of $N_1$ and $N_2$ forms a $t$-comotive  $N_1\otimes N_2:=N_1\otimes_{L[t]} N_2$ where $\sigma$ acts diagonally.
	\end{itemize}
\end{definition}

Let $\{\bn_1,\dots,\bn_r\}$ be an $L[t]$-basis for $N$ and let $\mathfrak{C}\in  \Mat_r(L[t])$ be such that 
\[
\sigma \cdot \begin{bmatrix}\bn_1\\ \vdots\\
\bn_r
\end{bmatrix}=\mathfrak{C}\begin{bmatrix}\bn_1\\ \vdots\\
\bn_r
\end{bmatrix}.
\] 
We say that $N$ is \textit{rigid analytically trivial} if there exists $\Psi\in \GL_r(\mathbb{T})$ such that 
\[
\Psi^{(-1)}=\mathfrak{C}\Psi.
\]
We call $\Psi$ \textit{a rigid analytic trivialization of $N$}.

In his unpublished notes, Anderson showed that there exists an equivalence between a subcategory of Anderson modules, namely the category of \textit{$t$-finite Anderson modules} and the category of $t$-comotives. We refer the reader to \cite[\S2.5.2]{HJ20} for full details of this correspondence and in what follows, we only review its one direction, namely, we attach a $t$-comotive to an Anderson module. Let $\phi$ be an Anderson module of dimension $d$ which is defined over $L$. We set $N_{\phi}:=\Mat_{1\times d}(L[\sigma])$. It is clear that $N_{\phi}$ is  a free and finitely generated $L[\sigma]$-module. It is also equipped with an  $L[t]$-module structure given by 
\[
\alpha t^i\cdot n:=\alpha  n \phi_{\theta^i}^{*}, \ \ n\in N_\phi, \ \ \alpha\in L.
\]
Hence, $N_{\phi}$ forms a left $L[t,\sigma]$-module. If $N_{\phi}$ is also a free $L[t]$-module of finite rank, we say that $\phi$ is \textit{a $t$-finite Anderson module}. We note that, due to Maurischat \cite[Thm. A]{Mau21},  being an abelian Anderson module is equivalent to being $t$-finite.

\subsubsection{Uniformizable abelian Anderson modules}\label{S:uniAndmod} In what follows, we  assume that $M$ is a $t$-motive over $\mathbb{C}_{\infty}$ whose corresponding abelian Anderson module, say $\phi$, is uniformizable. In this case, due to Namoijam and Papanikolas, there exists a matrix  $V\in \GL_r(\mathbb{C}_{\infty})$ whose construction is explicitly described in \cite[(4.4.12)]{NP21}. Let $\bm:=\{\bm_1,\dots.\bm_r\}:=(V^{\tr})^{(-1)}\{\widetilde{\bm}_1,\dots,\widetilde{\bm}_r\}$ be another $\mathbb{C}_{\infty}[t]$-basis for $M$. 
For each $1\leq \mu \leq r$ and $1\leq \ell \leq d$, there exist  elements $ \mathcal{C}_{\ell}^{[\mu]}\in \mathbb{C}_{\infty}[\tau] $ such that 
\begin{equation}\label{E:004}
\mathbf{m}_{\mu}=\sum_{\ell=1}^{d}\mathcal{C}_{\ell}^{[\mu]}\bb_{\ell}.
\end{equation}

We set 
\[
\Phi:=(V^{(-1)})^{-1}\Theta^{\tr}V\in \Mat_{r}(\mathbb{C}_{\infty}[t]).
\]
Then \eqref{E:003} implies
\begin{equation}\label{E:anotherbasis}
    \tau \cdot \mathbf{m}=\tau \cdot (V^{\tr})^{(-1)}\begin{bmatrix}\widetilde{\bm}_1\\ \vdots\\
\widetilde{\bm}_r
\end{bmatrix}=
V^{\tr}\tau \cdot \begin{bmatrix}\widetilde{\bm}_1\\ \vdots\\
\widetilde{\bm}_r
\end{bmatrix}=V^{\tr}\Theta ((V^{\tr})^{(-1)})^{-1}(V^{\tr})^{(-1)} \begin{bmatrix}\widetilde{\bm}_1\\ \vdots\\
\widetilde{\bm}_r
\end{bmatrix}=\Phi^{\tr}\mathbf{m}.
\end{equation}
Since $\bm$ is a $\mathbb{C}_{\infty}[t]$-basis for $M$, for each $1\leq \ell \leq d$ and $1\leq \mu \leq r$, there exist elements 
$
\mathcal{B}_{\mu}^{[\ell]}\in \mathbb{C}_{\infty}[t]
$ such that
\begin{equation}\label{E:008}
   \bb_{\ell}=\sum_{\mu=1}^{r}\mathcal{B}_{\mu}^{[\ell]}\mathbf{m}_{\mu}.
\end{equation}

We fix an $A$-module basis $\{\lambda_1,\dots,\lambda_{r}\}$ for  $\Lambda_{\phi}$ and for $1\leq \nu \leq r$, let $\mathcal{Y}_{\nu}(t)=(\mathcal{Y}_{\nu 1},\dots,\mathcal{Y}_{\nu d})^{\tr}$ be the Anderson generating function of $\phi$ with respect to $\lambda_\nu$. Consider the matrix
\[
\Upsilon:=\begin{pmatrix}
    \langle\widetilde{\bm}_1|\mathcal{Y}_{1}(t)\rangle & \cdots & \langle\widetilde{\bm}_1|\mathcal{Y}_r(t)\rangle\\
    \vdots & & \vdots \\
    \langle\widetilde{\bm}_r|\mathcal{Y}_1(t)\rangle & \cdots & \langle\widetilde{\bm}_r|\mathcal{Y}_r(t)\rangle
\end{pmatrix}\in \Mat_{r}(\mathbb{T}).
\]
Here, we again note that, letting $\mathfrak{U}:=(\mathcal{C}_{\ell}^{[\mu]})_{\mu \ell}\in \Mat_{r\times d}(\mathbb{C}_{\infty}[\tau])$, we realize $\widetilde{\mathbf{m}}_{\mu}$ as the element $(\widetilde{\mathcal{C}}_{1}^{[\mu]},\dots,\widetilde{\mathcal{C}}_{d}^{[\mu]})\in \Mat_{1\times d}(\mathbb{C}_{\infty}[\tau])\cong M$ where $(\widetilde{\mathcal{C}}_{1}^{[\mu]},\dots,\widetilde{\mathcal{C}}_{d}^{[\mu]})$ is the $\mu$-th row of $((V^{\tr})^{(-1)})^{-1}\mathfrak{U}$. 

Observe that 
\begin{equation}\label{E:007}
(\langle\widetilde{\mathbf{m}}_{\mu},\mathcal{Y}_{\nu}(t)\rangle)^{(1)}=\langle\tau \widetilde{\mathbf{m}}_{\mu},\mathcal{Y}_{\nu}(t)\rangle.
\end{equation}
 Thus, using \eqref{E:007} and \cite[(4.3.8), Lem. 4.3.9]{NP21}, we see that $\Upsilon\in \GL_{r}(\mathbb{T})$ and 
 \begin{equation}\label{E:tauactiongeneral}
   \Upsilon^{(1)}=\Theta \Upsilon.
 \end{equation}
In other words, $\Upsilon$ is a rigid analytic trivialization of $M$. On the other hand, by \cite[\S4.4]{NP21}, there exists a $\mathbb{C}_{\infty}[t]$-basis $\mathbf{n}:=\{\mathbf{n}_1,\dots,\mathbf{n}_r\}$ of $N_{\phi}$ such that 
\[
\sigma \mathbf{n}=\Phi \mathbf{n}.
\]
We set $\Psi:=((\Upsilon^{(1)})^{\tr}V)^{-1}\in \GL_{r}(\mathbb{T})$. Then, by the definition of $\Phi$ given in \S\ref{S:uniAndmod} and \eqref{E:tauactiongeneral}, we see that
\[
\Psi^{(-1)}=(V^{(-1)})^{-1}(\Upsilon^{\tr})^{-1}= (V^{(-1)})^{-1}\Theta^{\tr}((\Upsilon^{\tr})^{(1)})^{-1} =(V^{(-1)})^{-1}\Theta^{\tr}VV^{-1}((\Upsilon^{\tr})^{(1)})^{-1}=\Phi \Psi.
\]
Therefore, $\Psi$ is a rigid analytic trivialization for the $t$-comotive $N_{\phi}$ of $\phi$ (\cite[Prop. 4.67]{NP21}).

Set $\mathfrak{S}_{-1}:=\Id_r$, $\mathfrak{S}_0:=\Phi^{\tr}$ and for $n\geq 1$, consider $\mathfrak{S}_{n}:=(\Phi^{\tr})^{(n)}\cdots (\Phi^{\tr})^{(1)}\Phi^{\tr} $. For each $1\leq \mu \leq  r$, let $\mathfrak{e}_{\mu}\in \Mat_{r\times 1}(\mathbb{F}_q)$ be the $\mu$-th unit vector. Then, we have a $\mathbb{C}_{\infty}[t,\tau]$-module isomorphism $\widetilde{\varphi}_{\phi}:M\cong \Mat_{1\times d}(\mathbb{C}_{\infty}[\tau])\to \Mat_{1\times r}(\mathbb{C}_{\infty}[t])$ given by 
\begin{equation}\label{E:isomtmot}
\widetilde{\varphi}_{\phi}\left(\left(\sum_{i\geq 0}a_{1,i}\tau^i,\dots,\sum_{i\geq 0}a_{d,i}\tau^i\right)\right):=\sum_{1\leq \ell \leq d}\sum_{i\geq 0}a_{\ell,i}\sum_{\mu=1}^{r}(\mathcal{B}_{\mu}^{[\ell]})^{(i)}\mathfrak{e}_{\mu}^{\tr}\mathfrak{S}_{i-1}.
\end{equation}

Note that our ultimate goal in \S\ref{S:varphimap} is to construct a map $\varphi_{\phi}$ extending the isomorphism $\widetilde{\varphi}_{\phi}$ to a certain subset of $\Mat_{1\times d}(\mathbb{C}_{\infty}[[\tau]])$ containing the isomorphic copy $\Mat_{1\times d}(\mathbb{C}_{\infty}[\tau])$ of $M$.

\subsection{Examples} In the last subsection, we provide some examples of abelian and uniformizable Anderson modules as well as their corresponding $t$-motives and $t$-comotives.
\subsubsection{Drinfeld modules} \label{S:Drinfeld} \textit{A Drinfeld module of rank $r$ defined over $L$} is a one dimensional Anderson module $\phi:A\to L[\tau]$ defined by 
\begin{equation}\label{E:Drinfeldmod}
\phi_{\theta}=\theta+k_1\tau+\cdots +k_r\tau^r\in L[\tau], \ \ k_r\neq 0.
\end{equation}
 We further call the Drinfeld module $C$ of rank one given by $C_{\theta}:=\theta+\tau$ \textit{the Carlitz module}. 
 
 We call an $A$-module $\Lambda\subset \mathbb{C}_{\infty}$ \textit{an $A$-lattice of rank $r$} if it is free of rank $r$ over $A$ and it is discrete, namely its intersection with any ball in $\mathbb{C}_{\infty}$ of finite radius is finite. By Drinfeld \cite[\S5]{Dri74}, we know that there exists a one-to-one correspondence between Drinfeld modules of rank $r$ over $\mathbb{C}_{\infty}$ and $A$-lattices of rank $r$. We also refer the reader to \cite[\S4]{Goss} for more details on this correspondence. 

 We define $M_{\phi}:=L[\tau]$ and equip it with the $L[t]$-module structure given by 
\[
ct^i\cdot g\tau^j:=cg\tau^j\phi_{\theta^i},  \ \ c,g\in L.
\]
Note that $M_{\phi}$ is a left $L[t,\tau]$-module so that $(t-\theta)M_{\phi}\subset \tau M_{\phi}$. We define the matrix
\[
\Theta_{\phi}:=\begin{pmatrix}
&1& & &  \\
& & \ddots & & \\
& & & \ddots & \\
& & & &  1\\
\frac{t-\theta}{k_r}&-\frac{k_1}{k_r} & \dots & \dots & -\frac{k_{r-1}}{k_r}
\end{pmatrix}\in \GL_r(\mathbb{T})\cap \Mat_r(L[t]).
\]
In this case, we may choose an $L[t]$-basis $\bm=[1,\dots,\tau^{r-1}]^{\tr}\in \Mat_{r\times 1}(M_{\phi})$ for $M_{\phi}$ so that 
\[
\tau \cdot \bm=\Theta_{\phi} \bm.
\]
 Hence, $\phi$ is an abelian Anderson module. Moreover, $\{1\}$ forms an $L[\tau]$-basis for $M_{\phi}$. We also know that $\phi$ is always uniformizable and its period lattice $\Lambda_{\phi}$ is an $A$-lattice of rank $r$.

Let $\{\lambda_1,\dots,\lambda_r\}$ be an $A$-basis for $\Lambda_{\phi}$. For any $i\in \{1,\dots,r\}$, we denote by $f_{\lambda_i}(t)$  the Anderson generating function of $\phi$ with respect to $\lambda_i$. Consider the matrix 
\begin{equation}\label{E:Upsilon}
\Upsilon_{\phi}:=\begin{pmatrix} f_{\lambda_1}(t) & \cdots & \cdots & f_{\lambda_r}(t)\\
f_{\lambda_1}(t)^{(1)} & \cdots & \cdots & f_{\lambda_r}(t)^{(1)}\\
\vdots  & & & \vdots\\
f_{\lambda_1}(t)^{(r-1)} & \cdots & \cdots &f_{\lambda_r}(t)^{(r-1)}
\end{pmatrix}\in \Mat_{r\times r}(\mathbb{T}).
\end{equation}
By \cite[\S4.2]{Pel08} (see also Lemma \ref{L:AGF}(iii)), we know that $\Upsilon_{\phi}\in \GL_r(\mathbb{T})$ and moreover it satisfies 
\begin{equation}\label{E:RAT1}
\Upsilon_{\phi}^{(1)}=\Theta_{\phi} \Upsilon_{\phi}.
\end{equation}
Hence $M_{\phi}$ is rigid analytically trivial. Moreover, in this case, the matrix $V\in \GL_r(L)$ described in \S\ref{S:uniAndmod} becomes
\begin{equation}\label{E:V def}
V=\begin{pmatrix}
k_1&k_2^{(-1)} &k_3^{(-2)}& \dots &k_r^{(1-r)}\\
\vdots &\vdots&\vdots&  \iddots & \\
\vdots & \vdots& k_r^{(-2)} & & \\
\vdots &k_r^{(-1)} & & & \\
k_r &  & & & 
\end{pmatrix}\in \GL_r(L).
\end{equation}
Observe that  \begin{equation}\label{E:Phi def Drinfeld modules}
\Phi_{\phi}:=(V^{(-1)})^{-1}\Theta_{\phi}^{\tr}V=\begin{pmatrix}
&1& & &  \\
& & \ddots & & \\
& & & \ddots & \\
& & & &  1\\
\frac{t-\theta}{k_r^{(-r)}}&-\frac{k_1^{(-1)}}{k_r^{(-r)}} & \dots & \dots & -\frac{k_{r-1}^{(-(r-1))}}{k_r^{(-r)}}
\end{pmatrix}\in  \GL_r(\mathbb{T})\cap \Mat_r(L[t]).
\end{equation}

Next, we discuss the $t$-comotive of $\phi$. We define $N_{\phi}:=L[\sigma]$ equipped with the $L[t]$-module action given by 
\[
\alpha t^i\cdot \xi \sigma^j:=\alpha \xi \sigma^j\phi_{\theta^i}^{*},  \ \  \alpha,\xi\in L.
\]
Note that $N_{\phi}$ is free and finitely generated over $L[t]$ and $L[\sigma]$ satisfying $(t-\theta)N_{\phi}\subset \sigma N_{\phi}$.
We can let $\bn=[1,\dots,\sigma^{r-1}]^{\tr}\in \Mat_{r\times 1}(N_{\phi})$ so that 
\begin{equation}\label{E:tauact}
\sigma \cdot \bn=\Phi_{\phi} \bn.
\end{equation}
We see that $\bn$ forms an $L[t]$-basis for $N_{\phi}$. Moreover, $\{1\}$ forms an $L[\sigma]$-basis for $M_{\phi}$. 
We define  $\Psi_{\phi}:=V^{-1}((\Upsilon_{\phi}^{(1)})^{\tr})^{-1}\in \GL_r(\mathbb{T})$. By \S\ref{S:dualmot}, we know that
\begin{equation}\label{E:RAT2}
\Psi_{\phi}^{(-1)}=\Phi_{\phi}\Psi_{\phi}.
\end{equation}
Therefore $N_{\phi}$ is rigid analytically trivial.   
\subsubsection{The $k$-th tensor power $C^{\otimes k}$ of the Carlitz module}\label{S:Carlitztensor}
Assume that $k\in \mathbb{Z}_{\geq 1}$. We define the left $L[t,\tau]$-module 
\[
M_{C^{\otimes k}}:=M_{C}\otimes_{L[t]} \cdots \otimes_{L[t]} M_{C}=L[\tau]\otimes_{L[t]}\cdots \otimes_{L[t]} L[\tau]
\]
so that $\tau$ acts diagonally. Let $\bm_1$ be an  $L[t]$-basis for $M_C$. Then $\bm:=\bm_1\otimes \cdots \otimes \bm_1\in M_{C^{\otimes k}}$ is an $L[t]$-basis for $M_{C^{\otimes k}}$ so that 
\[
\tau \cdot \bm=(t-\theta)^k \bm.
\]
Furthermore, the set $\{\bm, (t-\theta)\bm,\dots,(t-\theta)^{k-1}\bm\}$ forms an $L[\tau]$-basis for $M_{C^{\otimes k}}$ and hence it is of dimension $k$ over $L[\tau]$. Therefore, $M_{C^{\otimes k}}\cong \Mat_{1\times k}(L[\tau])$ as $L[\tau]$-modules. The corresponding Anderson module $C^{\otimes k}:A\to \Mat_k(L[\tau])$ is given by
\[
C^{\otimes k}_{\theta}:=\begin{pmatrix}
		\theta&1& & \\
		& \ddots&\ddots & \\
		& & \ddots & 1\\
		& & & \theta  
		\end{pmatrix}+\begin{pmatrix}
		& &  & &\\
		& & & &\\
		& &  & &\\
		1&&&  &
		\end{pmatrix}\tau.
		\]
   Since $M_{C^{\otimes k}}$ is of finite rank over $L[t]$, $C^{\otimes k}$ is an abelian Anderson module.
   
Let us fix a $(q-1)$-st root of $-\theta$. \textit{The Anderson-Thakur function} $\omega_C$ is given by the infinite product 
\begin{equation}\label{D:omega_C}
\omega_C:=\sum_{n=0}^{\infty}\Exp_{C}\left(\frac{\tilde{\pi}}{\theta^{n+1}}\right)t^{n+1}=(-\theta)^{1/(q-1)}\prod_{j=0}^{\infty}\left(1-\frac{t}{\theta^{q^j}}\right)^{-1}\in \mathbb{T}^{\times}.
\end{equation}

We let $\mathbb{S}_{-1}:=1$, $\mathbb{S}_0:=t-\theta\in A[t]$ and for $n\geq 1$, define  
\begin{equation}\label{E:sn}
\mathbb{S}_n:=(t-\theta^{q^n})\cdots (t-\theta)\in A[t].
\end{equation}
Note that, for $\ell\geq 1$, Lemma \ref{L:AGF}(iii) implies
\begin{equation}\label{E:AGF4}
(\omega_C^k)^{(\ell)}=\mathbb{S}_{\ell-1}^k\omega_C^k.
\end{equation}
Therefore, letting $\ell=1$ above, we see that $\omega_C^k$ is a rigid analytic trivialization for $M_{C^{\otimes k}}$, implying that $M_{C^{\otimes k}}$ is rigid analytically trivial. Thus, by \cite[Thm. 4]{And86}, $C^{\otimes k}$ is uniformizable and its period lattice $\Lambda_{C^{\otimes k}}$ is of rank $1$.

We now discuss the $t$-comotive of $C^{\otimes k}$. Consider the $L[t,\sigma]$-module given by
\[
N_{C^{\otimes k}}:=N_{C}\otimes_{L[t]} \cdots \otimes_{L[t]} N_{C}=L[\sigma]\otimes_{L[t]}\cdots \otimes_{L[t]} L[\sigma]
\]
where $\sigma$ acts diagonally.  Let $\bn_1$ be an $L[t]$-basis for $N_C$. Then $\bn=\bn_1\otimes \cdots \otimes \bn_1\in N_{C^{\otimes k}}$ is an $L[t]$-basis for $N_{C^{\otimes k}}$ so that 
\[
\tau \bn=(t-\theta)^k \bn.
\]
Furthermore, the set $\{\bn, (t-\theta)\bn,\dots,(t-\theta)^{k-1}\bn\}$ forms an $L[\tau]$-basis for $N_{C^{\otimes k}}$ and hence it is of dimension $k$ over $L[\sigma]$. Therefore, $N_{C^{\otimes k}}\cong \Mat_{1\times k}(L[\sigma])$ as $L[\sigma]$-modules.

We now define the element 
$
\Omega(t):=(\omega_C^{(1)})^{-1}\in \mathbb{T}^{\times}.
$
Note that 
\begin{equation}\label{E:RAT3}
(\Omega(t)^k)^{(-1)}=(t-\theta)^k\Omega(t)
\end{equation}
implying that $\Omega(t)^k$ is a rigid analytic trivialization of $N_{C^{\otimes k}}$. Therefore $N_{C^{\otimes k}}$ is rigid analytically trivial.
\section{Proof of Theorem \ref{T:Intro log formula}}\label{S:Abelian And Modules}
In this section we will give the proof of Theorem \ref{T:Intro log formula}. We start out with a general formula related to Anderson generating functions. Then we describe the motivic map $\delta_{1,\zz}^{M}$ of the second author. Following that, we give the proof of Theorem \ref{T:Intro log formula}. Throughout this section, we fix $M$ to be an abelian $t$-motive of rank $r$ so that its corresponding Anderson module $\phi$ is uniformizable. We recall the notation from \S\ref{S:uniAndmod} and let $L=\mathbb{C}_{\infty}$.

\subsection{Anderson generating function formulas}
Recall that $\mathbf{b}=\{\bb_1,\dots,\bb_d\}$ is the $\mathbb{C}_{\infty}[\tau]$-basis of $M$ as in \eqref{E:002}. Then the corresponding Anderson module
$\phi:A\to \Mat_{d}(\mathbb{C}_{\infty}[\tau])$ may be defined by 
\[
\phi_{\theta}=\mathfrak{M}.
\]
 We again fix an $A$-module basis $\{\lambda_1,\dots,\lambda_{r}\}$ for  $\Lambda_{\phi}$ and for $1\leq \nu \leq r$, let $\mathcal{Y}_{\nu}(t)=(\mathcal{Y}_{\nu 1},\dots,\mathcal{Y}_{\nu d})^{\tr}$ be the Anderson generating function of $\phi$ with respect to $\lambda_\nu$. By Lemma \ref{L:AGF}(iii), we have 
\begin{equation}\label{E:001}
t\begin{pmatrix}
    \mathcal{Y}_{\nu 1}\\
    \vdots\\
    \vdots\\
    \mathcal{Y}_{\nu d}
\end{pmatrix}=\phi_{\theta}\begin{pmatrix}
    \mathcal{Y}_{\nu 1}\\
    \vdots\\
    \vdots\\
    \mathcal{Y}_{\nu d}
\end{pmatrix}=\mathfrak{M}\begin{pmatrix}
    \mathcal{Y}_{\nu 1}\\
    \vdots\\
    \vdots\\
    \mathcal{Y}_{\nu d}
\end{pmatrix}.
\end{equation}

For each $1\leq \ell \leq d$ and $1\leq \mu \leq r$, we now recall the elements $\mathcal{C}_{\mu}^{[\ell]}$ and $\mathcal{C}_{\ell}^{[\mu]}$ defined in \S\ref{S:Abelian And Modules}. Combining \eqref{E:004} and \eqref{E:008}, for each $1\leq \ell \leq d$, we thus have a system of equations in $\Mat_{1\times d}(\mathbb{C}_{\infty}[\tau])\cong M$ given by
\[
\bb_{\ell}=\left( \mathcal{B}^{[\ell]}_{1}\mathcal{C}_{1}^{[1]}+\cdots+\mathcal{B}_r^{[\ell]}\mathcal{C}_1^{[r]} \right)\cdot \bb_1+\cdots+\left( \mathcal{B}^{[\ell]}_{1}\mathcal{C}_{d}^{[1]}+\cdots+\mathcal{B}_r^{[\ell]}\mathcal{C}_{d}^{[r]} \right)\cdot \bb_d.
\]
Using the $t$-action on $\bb=\{\bb_1,\dots,\bb_d\}$ given in \eqref{E:002}, we then obtain elements $\mathfrak{B}_{1},\dots,\mathfrak{B}_{d}\in \mathbb{C}_{\infty}[\tau]$ such that 
\[
\bb_{\ell}=\mathfrak{B}_1\cdot \bb_1+\cdots+\mathfrak{B}_{d}\cdot \bb_d.
\]
Since $\{\bb_1,\dots,\bb_{d}\}$ is $\mathbb{C}_{\infty}[\tau]$-linearly independent in $M$, we have
\begin{equation}\label{E:linindep}
\mathfrak{B}_{i}=\begin{cases} 1 & \text{ if } i=\ell\\
0 & \text{ otherwise. }    
\end{cases}
\end{equation}
Now let us set 
\begin{multline}\label{E:009}
\mathfrak{h}_{\nu \ell}:=\sum_{\mu=1}^{r}\mathcal{B}_{\mu}^{[\ell]}\langle\mathbf{m}_\mu,\mathcal{Y}_{\nu}(t)\rangle=\sum_{\mu=1}^{r}\mathcal{B}_{\mu}^{[\ell]}\sum_{\ell=1}^{d}\mathcal{C}_{\ell}^{[\mu]}\mathcal{Y}_{\nu \ell}\\
=\left( \mathcal{B}^{[\ell]}_{1}\mathcal{C}_{1}^{[1]}+\cdots+\mathcal{B}_r^{[\ell]}\mathcal{C}_1^{[r]} \right)\cdot \mathcal{Y}_{\nu 1}+\cdots+\left( \mathcal{B}^{[\ell]}_{1}\mathcal{C}_{d}^{[1]}+\cdots+\mathcal{B}_r^{[\ell]}\mathcal{C}_{d}^{[r]} \right)\cdot \mathcal{Y}_{\nu d}\in \mathbb{T}.
\end{multline}
Here, we realize $\mathbf{m}_{\mu}$ as the element $(\mathcal{C}_{1}^{[\mu]},\dots,\mathcal{C}_{d}^{[\mu]})\in \Mat_{1\times d}(\mathbb{C}_{\infty}[\tau])\cong M$. 
Since, by \eqref{E:002} and \eqref{E:001}, the $t$-action on $\{\bb_1,\dots,\bb_d\}$ and $\{\mathcal{Y}_{\nu 1},\dots,\mathcal{Y}_{\nu d}\}$ are the same, we obtain
\[
\mathfrak{h}_{\nu \ell}=\mathfrak{B}_1\cdot \mathcal{Y}_{\nu 1}+\cdots+\mathfrak{B}_{d}\cdot \mathcal{Y}_{\nu d}
\]
implying, by \eqref{E:linindep}, that
\begin{equation}\label{E:0015}
\mathcal{Y}_{\nu \ell}=\mathfrak{h}_{\nu \ell}=\sum_{\mu=1}^{r}\mathcal{B}_{\mu}^{[\ell]}\langle\mathbf{m}_\mu,\mathcal{Y}_{\nu}(t)\rangle.
\end{equation}

 Let $
\mathcal{R}_{\phi}:=(\alpha_1,\dots,\alpha_{d})^{\tr}\in \mathbb{R}_{\geq 0}^{d}$
be the radius of convergence of $\Log_{\phi}$, namely, for $\zz=(z_1,\dots,z_d)\in \mathbb{C}_{\infty}^d$, $\Log_{\phi}(\zz)$ converges in $\mathbb{C}_{\infty}^d$ if and only if for each $1\leq \ell \leq d$, we have $|z_{\ell}|<\alpha_{\ell}$.  We continue with a crucial lemma that will be useful to state our main result in this section. 

\begin{lemma}\label{L:integerm}
 There exists a unique smallest non-negative integer $\mathfrak{m}_{\phi}$ such that if 
\[
\frac{1}{\theta^{\mathfrak{m}_{\phi}}}\Exp_{\phi}\left(\dmat_{\phi}[\theta]^{-i}\lambda_{\nu}\right)=(z_1,\dots,z_d)^{\tr},
\]
then for each $1\leq \nu \leq r$ and $i\geq 1$, we have 
\[
|z_{\ell}|<\alpha_{\ell}, \ \ 1\leq \ell \leq d.
\]
\end{lemma}
\begin{proof}
By \cite[Lem. 5.3]{HJ20}, we know that there exists a subdomain $\mathcal{V}\subset \mathbb{C}_{\infty}^d$ so that $\Exp_{\phi}$ maps $\mathcal{V}$ $|\cdot|$-isometrically to itself. Thus, since $\Log_{\phi}$ is the formal inverse of $\Exp_{\phi}$, there exists a smallest integer $i_0$ so that whenever $i>i_0$, we have
   \[
\left|\Exp_{\phi}\left(\dmat_{\phi}[\theta]^{-i}\lambda_{\nu}\right)\right|=\left| \dmat_{\phi}[\theta]^{-i}\lambda_{\nu} \right|
\]
  for  $1\leq \nu \leq r$ and hence $\Exp_{\phi}\left(\dmat_{\phi}[\theta]^{-i}\lambda_{\nu}\right)$ lies in the domain of convergence of $\Log_{\phi}$.
    If $i_0=0$, then we choose $\mathfrak{m}_{\phi}=0$. If not, then we choose $\mathfrak{m}_{\phi}$ to be the smallest integer such that if we set 
\[
\Exp_{\phi}\left(\dmat_{\phi}[\theta]^{-i}\lambda_{\nu}\right)=(\tilde{z}_1,\dots,\tilde{z}_d)^{\tr},
\]
then for each $1\leq \nu \leq r$ and $1\leq i\leq i_0+1$, we have 
\[
|\tilde{z}_{\ell}|<|\theta|^{\mathfrak{m}_{\phi}}\alpha_{\ell}, \ \ 1\leq \ell \leq d.
\]
This finishes the proof of the first assertion. On the other hand, note that since 
\[
||\theta^{-\mathfrak{m}_{\phi}}\mathcal{Y}_{\nu}(t)||=\max_{i\geq 0}\left| \theta^{-\mathfrak{m}_{\phi}}\Exp_{\phi}\left(\dmat_{\phi}[\theta]^{-i-1}\lambda_{\nu}\right) \right|,
\]
if we let $\theta^{-\mathfrak{m}_{\phi}}\mathcal{Y}_{\nu}(t)=(g_1,\dots,g_d)^{\tr}\in \mathbb{T}^d$, then by the first assertion, we have
\[
||g_{\ell}||<\alpha_{\ell}, \ \ \ \ \ 1\leq \ell \leq d
\]
finishing the proof of the second assertion.


    \end{proof}
Next, we set $\mathcal{P}:=\left( \sum_{i\geq 0}a_{1,i}\tau^i,\dots,\sum_{i\geq 0}a_{d,i}\tau^i\right)\in \Mat_{1\times d}(\mathbb{C}_{\infty}[[\tau]])$ and
for each $1\leq \nu \leq  r$ and $\mathcal{Y}_{\nu}(t)=(\mathcal{Y}_{\nu 1},\dots,\mathcal{Y}_{\nu d})^{\tr}$, define
\[
\mathcal{P}(\mathcal{Y}_{\nu}(t)):=\sum_{1\leq \ell \leq d}\sum_{i\geq 0}a_{\ell,i}\mathcal{Y}_{\nu \ell}^{(i)}\in \mathbb{C}_{\infty}[[t]].
\]
Recall that for $1\leq \mu \leq r$, $\mathfrak{e}_{\mu}\in \Mat_{r\times 1}(\mathbb{F}_q)$ is the $\mu$-th unit vector.
\begin{proposition} \label{P:eqtngen} For each $1\leq \nu \leq  r$, assume that $\mathcal{P}(\mathcal{Y}_{\nu}(t))$ converges in $\mathbb{T}$. Then, we have
    \begin{equation*}
    \left[\mathcal{P}(\mathcal{Y}_1(t)),\dots,\mathcal{P}(\mathcal{Y}_r(t))\right](\Psi^{\tr})^{(-1)}    =\sum_{1\leq \ell \leq d}\sum_{i\geq 0}a_{\ell,i}\sum_{\mu=1}^{r}(\mathcal{B}_{\mu}^{[\ell]})^{(i)}\mathfrak{e}_{\mu}^{\tr}\mathfrak{S}_{i-1}.
\end{equation*}
\end{proposition}
\begin{proof}
Observe that
\begin{multline*}
\left[\mathcal{P}(\mathcal{Y}_1(t)),\dots,\mathcal{P}(\mathcal{Y}_r(t))\right]=\left[\sum_{1\leq \ell \leq d}\sum_{i\geq 0}a_{\ell,i}\mathcal{Y}_{1 \ell}^{(i)},\dots,\sum_{1\leq \ell \leq d}\sum_{i\geq 0}a_{\ell,i}\mathcal{Y}_{r \ell}^{(i)}\right] \\
=\left[\sum_{1\leq \ell \leq d}\sum_{i\geq 0}a_{\ell,i}\left(\sum_{\mu=1}^{r}\mathcal{B}_{\mu}^{[\ell]}\langle\mathbf{m}_\mu,\mathcal{Y}_1(t)\rangle\right)^{(i)},\dots, \sum_{1\leq \ell \leq d}\sum_{i\geq 0}a_{\ell,i}\left(\sum_{\mu=1}^{r}\mathcal{B}_{\mu}^{[\ell]}\langle\mathbf{m}_\mu,\mathcal{Y}_r(t)\rangle\right)^{(i)}\right]\\
=\sum_{1\leq \ell \leq d}\sum_{i\geq 0}a_{\ell,i}\left(\sum_{\mu=1}^{r}\mathcal{B}_{\mu}^{[\ell]}\mathfrak{e}_{\mu}^{\tr}(V^{\tr})^{(-1)}\Upsilon\right)^{(i)}=\sum_{1\leq \ell \leq d}\sum_{i\geq 0}a_{\ell,i}\sum_{\mu=1}^{r}(\mathcal{B}_{\mu}^{[\ell]})^{(i)}\mathfrak{e}_{\mu}^{\tr}((V^{\tr})^{(-1)}\Upsilon)^{(i)}\\
=\sum_{1\leq \ell \leq d}\sum_{i\geq 0}a_{\ell,i}\sum_{\mu=1}^{r}(\mathcal{B}_{\mu}^{[\ell]})^{(i)}\mathfrak{e}_{\mu}^{\tr}\mathfrak{S}_{i-1}(V^{\tr})^{(-1)}\Upsilon
=\sum_{1\leq \ell \leq d}\sum_{i\geq 0}a_{\ell,i}\sum_{\mu=1}^{r}(\mathcal{B}_{\mu}^{[\ell]})^{(i)}\mathfrak{e}_{\mu}^{\tr}\mathfrak{S}_{i-1}((\Psi^{\tr})^{(-1)})^{-1}.
\end{multline*}
Here, the second equality follows from \eqref{E:0015}, the third equality from the definition of $\Upsilon$, the equality $[\bm_1,\dots,\bm_r]^{\tr}=(V^{\tr})^{(-1)}[\widetilde{\bm}_1,\dots,\widetilde{\bm}_r]^{\tr}$ as well as the $\mathbb{C}_{\infty}$-linearity of $\langle\cdot | \cdot\rangle$ in the first entry. Finally, the fifth and sixth equality follow from the fact that 
\[
\Phi^{\tr}((V^{\tr})^{(-1)}\Upsilon)=\Phi^{\tr}((\Psi^{-1})^{\tr})^{(-1)}=(\Psi^{-1})^{\tr}=((V^{\tr})^{(-1)}\Upsilon)^{(1)}.
\]
Thus, we obtain 
\[
    \left[\mathcal{P}(\mathcal{Y}_1(t)),\dots,\mathcal{P}(\mathcal{Y}_r(t))\right](\Psi_{\phi}^{\tr})^{(-1)}    =\sum_{1\leq \ell \leq d}\sum_{i\geq 0}a_{\ell,i}\sum_{\mu=1}^{r}(\mathcal{B}_{\mu}^{[\ell]})^{(i)}\mathfrak{e}_{\mu}^{\tr}\mathfrak{S}_{i-1}
\]
as desired.
\end{proof}
\subsection{The map $\varphi$}\label{S:varphimap}

For convenience, we recall 
    \[
\mathcal{Y}_{\nu}(t)=\sum_{i=0}^{\infty}\Exp_{\phi}\left(\dmat_{\phi}[\theta]^{-(i+1)}\lambda_{\nu}\right)t^i=(\mathcal{Y}_{\nu 1},\dots,\mathcal{Y}_{\nu d})^{\tr}, \ \ \ \ \ \ 1\leq \nu \leq r.
    \] 
We now define
\begin{multline*}
\mathbb{M}_{\phi}:=\{\left(\sum_{n=0}^{\infty}a_{1,n}\tau^n,\dots,\sum_{n=0}^{\infty}a_{d,n}\tau^n\right)\in \Mat_{1\times d}(\mathbb{C}_{\infty}[[\tau]])\ \ | \ \ \text{ for each } 1\leq \ell \leq d\ \ \\
|a_{\ell,n}|||\mathcal{Y}_{\nu \ell}^{(n)}||\to 0 
\text{ as } n\to \infty \}.
\end{multline*}
 Let $\mathcal{G}=\left(\sum_{n=0}^{\infty}a_{1,n}\tau^n,\dots,\sum_{n=0}^{\infty}a_{d,n}\tau^n\right)\in \mathbb{M}_{\phi}$. Thus, by our condition on elements in $\mathbb{M}_{\phi}$, we see that  
\[
\mathfrak{a}_{\nu}:=\sum_{n=0}^{\infty}a_{1,n}\mathcal{Y}_{\nu 1}^{(n)}+\cdots+\sum_{n=0}^{\infty}a_{d,n}\mathcal{Y}_{\nu d}^{(n)}\in \mathbb{C}_{\infty}[[t]]
\]
converges in $\mathbb{T}$ for each $1\leq \nu \leq r$. We also note that $\Log_{\phi}\frac{1}{\theta^{\mathfrak{m}_{\phi}}}\in \mathbb{M}_{\phi}$.
We equip $\mathbb{M}_{\phi}$ with the norm $|\cdot|_{\phi}$ defined by 
\[
|\mathcal{G}|_{\phi}:=\max\{||a_{\ell,n}\mathcal{Y}_{\nu\ell}^{(n)}|| \ \ | \ \ n\geq 0,\ \ \ 1\leq \ell \leq d,  \ \ 1\leq \nu \leq r\}.
\]
Observe that $(\mathbb{M}_{\phi},|\cdot|_{\phi})$ is a normed $\mathbb{C}_{\infty}$-vector space. Furthermore, it is easy to see that the normed $\mathbb{C}_{\infty}$-vector space $(M,|\cdot|_{\phi})$ is dense in $(\mathbb{M}_{\phi},|\cdot|_{\phi})$.

Consider the map
$\varphi_{\phi}:\mathbb{M}_{\phi}\to \Mat_{1\times r}(\mathbb{C}_{\infty}[[t]])$ given by   
\[
\varphi_{\phi}\left(\left(\sum_{n=0}^{\infty}a_{1,n}\tau^n,\dots,\sum_{n=0}^{\infty}a_{d,n}\tau^n\right)\right):=\sum_{1\leq \ell \leq d}\sum_{n=0}^{\infty}a_{\ell,n}\sum_{\mu=1}^{r}(\mathcal{B}_{\mu}^{[\ell]})^{(n)}\mathfrak{e}_{\mu}^{\tr}\mathfrak{S}_{n-1}.
\]

\begin{theorem}\label{T:tenwelinj} The map $\varphi_{\phi}$ is a well-defined map that converges in $(\Mat_{1\times r}(\mathbb{T}),||\cdot||)$. Furthermore, it is continuous and injective.
\end{theorem}
\begin{proof} Let $\mathcal{G}\in \mathbb{M}_{\phi}$ be as above. Since for each $1\leq \nu \leq r$, $\mathfrak{a}_{\nu}$ converges in $\mathbb{T}$, 
by Proposition \ref{P:eqtngen}, we see that
\[
\varphi_{\phi}(\mathcal{G})=\sum_{1\leq \ell \leq d}\sum_{n= 0}^{\infty}a_{\ell,n}\sum_{\mu=1}^{r}(\mathcal{B}_{\mu}^{[\ell]})^{(n)}\mathfrak{e}_{\mu}^{\tr}\mathfrak{S}_{n-1}=\left[\mathfrak{a}_1,\dots,\mathfrak{a}_r\right](\Psi^{\tr})^{(-1)}
\]
and hence $\varphi_{\phi}$ is a well-defined map that converges in $\Mat_{1\times r}(\mathbb{T})$.
On the other hand, we obtain 
\[
||\varphi_{\phi}(\mathcal{G})||\leq ||\left[\mathfrak{a}_1,\dots,\mathfrak{a}_r\right]||||(\Psi^{\tr})^{(-1)}||\leq |\mathcal{G}|_{\phi}||(\Psi^{\tr})^{(-1)}||,
\]
implying that $\varphi_{\phi}$ is a continuous map.

We show that $\varphi_{\phi}$ is injective. Assume to the contrary that $ \mathcal{G}$ is not identically zero and satisfies
\[
\varphi_{\phi}(\mathcal{G})=\left[\mathfrak{a}_1,\dots,\mathfrak{a}_r\right](\Psi^{\tr})^{(-1)}=[0,\dots,0]
\]
    or, using the fact that $\Psi\in \GL_r(\mathbb{T})$, in other words
    \begin{equation}\label{E:inj2}
    \left[\mathfrak{a}_1,\dots,\mathfrak{a}_r\right]=[0,\dots,0].
    \end{equation}
    Hence \eqref{E:inj2} implies that for each $1\leq \nu \leq r$ and $i\geq 0$
\begin{equation}\label{E:sum vanishes}
\sum_{n=0}^{\infty}(a_{1,n},\dots,a_{d,n})\Exp_{\phi}\left(\dmat_{\phi}[\theta]^{-(i+1)}\lambda_\nu\right)^{(n)}=0.  
\end{equation}
On the other hand, since $\Exp_{\phi}$ is an isometric map on some neighborhood of $0$, there exists a positive integer $i_0$ such that $\dmat_{\phi}[\theta]^{-(i+1)}\lambda_\nu$ lies in the domain of convergence of $\Log_{\phi}$  for each $1\leq \mu \leq r$ and  $i\geq i_0$. Thus, we have
    \[
    \left|(a_{1,n},\dots,a_{d,n})\Exp_{\phi}\left(\dmat_{\phi}[\theta]^{-(i+1)}\lambda_\nu\right)^{(n)}\right|=\left|(a_{1,n},\dots,a_{d,n})\left(\dmat_{\phi}[\theta]^{-(i+1)}\lambda_\nu\right)^{(n)}\right|.
   \]
Let us fix such an $i_0$ and assume $i>i_0$. By the calculation in \cite[(51)]{Gre22} (see also \cite[Rem. 3.3.3]{NP21}), we see that for sufficiently large $i$, the norm
\[\left|(a_{1,n},\dots,a_{d,n})\left(\dmat_{\phi}[\theta]^{-(i+1)}\lambda_\nu\right)\right|\]
is dominated by $|\theta|^{-(i+1)}$. Thus, for fixed $0<n_1<n_2$, by multiplying both sides by $|\theta|^{q^{n_1}(i+1)}$ we have that for sufficiently large $i$
\begin{equation}\label{E:normlessthan}
\left|(a_{1,n_1},\dots,a_{d,n_1})\left(\dmat_{\phi}[\theta]^{-(i+1)}\lambda_\nu\right)^{(n_1)}\right|>\left|(a_{1,n_2},\dots,a_{d,n_2})\left(\dmat_{\phi}[\theta]^{-(i+1)}\lambda_\nu\right)^{(n_2)}\right|.
\end{equation}
In particular, if \eqref{E:normlessthan} holds for some $i$, then it also holds for all $m>i$ by the same reasoning. Further, if for some $i_0$ we have that
\[\left|(a_{1,n_1},\dots,a_{d,n_1})\left(\dmat_{\phi}[\theta]^{-(i_0+1)}\lambda_\nu\right)^{(n_1)}\right|=\left|(a_{1,n_2},\dots,a_{d,n_2})\left(\dmat_{\phi}[\theta]^{-(i_0+1)}\lambda_\nu\right)^{(n_2)}\right|,\]
and both are not identically zero, then for $m>i_0$, we have
\begin{equation}\label{E:samenormdecrease}
\left|(a_{1,n_1},\dots,a_{d,n_1})\left(\dmat_{\phi}[\theta]^{-(m+1)}\lambda_\nu\right)^{(n_1)}\right|>\left|(a_{1,n_2},\dots,a_{d,n_2})\left(\dmat_{\phi}[\theta]^{-(m+1)}\lambda_\nu\right)^{(n_2)}\right|.
\end{equation}

Thus, by the non-archimedean property of the norm $|\cdot|$, since \eqref{E:sum vanishes} for $i=0$ converges and equals 0, if any $a_{j,i}\neq 0$, then there must be a finite number of terms of \eqref{E:sum vanishes} which have the same norm. Then, by applying \eqref{E:samenormdecrease} a finite number of times, we conclude that for a sufficiently large $i$ there is a unique term of \eqref{E:sum vanishes} which has the highest norm (in fact, we can take $i$ large enough such that it must be the first non-vanishing term). But by the non-archimedean property of the norm, this would imply that \eqref{E:sum vanishes} has nonzero norm, which is a contradiction. This implies that $a_{\ell,n}=0$ for all $\ell,n$, so $\varphi_\phi$ is injective.

\end{proof}

Before we prove the main result of this section, we define a fundamental map whose construction is due to the second author \cite{Gre24}. Recall that $M\cong \Mat_{1\times d}(\mathbb{C}_{\infty}[\tau])$ and fix $\zz=(z_1,\dots,z_d)^{\tr} \in \mathbb{C}_{\infty}^d$. We define $\delta_{1,\zz}^M:M\to \mathbb{C}_{\infty}$ by 
\begin{equation}\label{D:delta_1 map}
\delta_{1,\zz}^M(m):=m\zz:=m_1\cdot z_1+\dots+m_d \cdot z_d, \ \ m=[m_1,\dots,m_d]\in \Mat_{1\times d}(\mathbb{C}_{\infty}[\tau]).
\end{equation}
 We also define $M_{\zz}$ to be the set of elements $(\mathfrak{c}_1,\dots,\mathfrak{c}_d)$ where, for each $i\in \{1,\dots,d\}$,  $\mathfrak{c}_i=\sum_{j=0}^{\infty}c_{i,j}\tau^j\in \mathbb{C}_{\infty}[[\tau]]$ satisfies 
$
\left(c_{1,\mu}\tau^\mu,\dots,c_{d,\mu}\tau^\mu \right)\zz\to 0 $ as $\mu\to \infty$. Then we extend the map $\delta_{1,\zz}^M$ to $M_{\zz}$ by defining $\delta_{1,\zz}^{M}:M_{\zz}\to \mathbb{C}_{\infty}$ as
\[
\delta_{1,\zz}^M(\tilde{m}):=\lim_{\mu\to \infty}\delta_{1,\zz}([\mathfrak{a}^{\mu}_1,\dots,\mathfrak{a}^{\mu}_d])
\]
where
$
\tilde{m}=\left[\sum_{j=0}^{\infty}a_{1,j}\tau^j,\dots, \sum_{j=0}^{\infty}a_{d,j}\tau^j\right]
$
and 
$\mathfrak{c}_i^{\mu}:=\sum_{j=0}^{\mu}c_{i,j}\tau^j$. Finally, we extend $\delta_{1,\bz}^{M}$ to vectors in $M_\bz^d$ by acting coordinate-wise. We refer the reader to \cite[Def. 2.19]{Gre24} for more details on this extension.

\begin{proof}[Proof of Theorem \ref{T:Intro log formula}] For each $1\leq \ell \leq d$ and $1\leq \nu \leq r$, recall $\mathcal{P}_{\ell}(\mathcal{Y}_{\nu}(t))\in \mathbb{T}$ defined in \eqref{D:mathcal P}. Note, by Proposition \ref{P:eqtngen}, that 
\begin{equation*}
\varphi_{\phi}\left(\sum_{n\geq 0}\frac{1}{\theta^{q^n\mathfrak{m}_{\phi}}}L_{\ell 1}^{[n]}\tau^n,\dots,\sum_{n\geq 0}\frac{1}{\theta^{q^n\mathfrak{m}_{\phi}}}L_{\ell d}^{[n]}\tau^n\right)=\left[\mathcal{P}_{\ell}(\mathcal{Y}_1(t)),\dots,\mathcal{P}_{\ell}(\mathcal{Y}_r(t))\right](\Psi^{\tr})^{(-1)}.
\end{equation*}
Recall that for any $\zz\in \mathbb{C}_{\infty}^d$ lying the domain of convergence of $\Log_{\phi}\theta^{-\mathfrak{m}_{\phi}}$, we set $\mathcal{M}_{\zz}=\delta_{1,\zz}^{M}\circ \varphi_{\phi}^{-1}$. Thus, we have
\begin{equation}\label{E:formulaproof}
\mathfrak{p}_{\ell}(\Log_{\phi}(\theta^{-\mathfrak{m}_{\phi}}\zz))=
\mathcal{M}_{\zz}\left([\mathcal{P}_{\ell}(\mathcal{Y}_1(t)),\dots, \mathcal{P}_{\ell}(\mathcal{Y}_r(t))](\Psi^{\tr})^{(-1)}\right).
\end{equation}
Note, by Lemma \ref{L:integerm}, if $\mathfrak{m}_{\phi}=0$, then $\mathcal{Y}_{\nu}(t)$ lies in the domain of the convergence of $\Log_{\phi}:\mathcal{D}'\to \mathbb{T}^{d}$. Moreover, if for $i\geq 1$, each $\dmat_{\phi}[\theta]^{-i}\lambda_{\nu}$ lies in the domain of the convergence of $\Log_{\phi}:\mathcal{D}\to \mathbb{C}_{\infty}^d$ and if $\mathfrak{t}$ is a tractable coordinate of $\phi$, by Lemma \ref{L:AGF}(i), we then have 
\begin{equation}\label{E:eqtngen2}
\mathfrak{p}_{\mathfrak{t}}(\Log_{\phi}(\mathcal{Y}_{\nu}(t)))=\mathfrak{p}_{\mathfrak{t}}\left(\Log_{\phi}\left(\Exp_{\phi}\left(\dmat_{\phi}[\theta]^{-i}\lambda_{v}\right)\right)\right)=\frac{\lambda_{\nu \mathfrak{t}}}{\theta-t}.
\end{equation} Thus, by \eqref{E:formulaproof} and \eqref{E:eqtngen2}, we obtain
\[
\mathfrak{p}_{\mathfrak{t}}(\Log_{\phi}(\zz))=\mathcal{M}_{\zz}\left(\frac{1}{\theta-t}[\lambda_{1 \mathfrak{t}},\dots, \lambda_{r \mathfrak{t}}](\Psi^{\tr})^{(-1)}\right)
\]
as desired.
\end{proof}

\section{Drinfeld modular forms and dual Goss $L$-functions of Drinfeld modules} \label{S:applications}
In this section, our goal is to apply Theorem \ref{T:Intro log formula} in the case of Drinfeld modules defined over $A$ and relate them to the special values of dual Goss $L$-functions. Later on, we interpret our formulas in terms of modular matrix-valued functions with respect to $\rho^{*}$ described in \S\ref{S:intro1.2} and prove Theorem \ref{T:Intromodform}.
\subsection{Dual Goss $L$-function of Drinfeld modules} Let $\phi$ be the Drinfeld module of rank $r$ given by 
\begin{equation}\label{E:drinmod}
\phi_{\theta}=\theta+k_1\tau+\dots+k_r\tau^r\in \mathbb{C}_{\infty}[\tau].
\end{equation}
We also let $\{\lambda_1,\dots,\lambda_r\}$ be a fixed generating set for the period lattice of $\phi$. 

Our next theorem generalizes \cite[Cor. 1.3]{GG25} to Drinfeld modules of arbitrary rank defined over $A$.
\begin{theorem}\label{T:modformslvalues} Let $\phi$ be a Drinfeld module of rank $r$ defined over $A$.
\begin{itemize}
\item[(i)] There exist an element $\mathfrak{f}_{\phi}\in K_{\infty}^{\times}$, a non-negative integer $N(\phi)$ and an element $\mathfrak{c}_{\phi}\in K^{\times}$ such that setting $\zz:=\theta^{\mathfrak{m}_{\phi}}\Exp_{\phi}(\mathfrak{f}_{\phi}/\theta^{N(\phi)})$, we have
\[
L(\phi^{\vee},0)=\mathfrak{c}_{\phi}\mathcal{M}_{\zz}\left(
\left[\Log_{\phi}\left(\frac{f_{\lambda_1}(t)}{\theta^{m_{\phi}}}\right),\dots,\Log_{\phi}\left(\frac{f_{\lambda_r}(t)}{\theta^{m_{\phi}}}\right)\right]
(\Psi_{\phi}^{\tr})^{(-1)}\right).
\]
\item[(ii)] Let $r=2$ and assume that for $i=1,2$, we have $\deg_{\theta}(k_i)<q^i$. Then $\zz=1$ and
\[
L(\phi^{\vee},0)=\mathfrak{c}_{\phi}\mathcal{M}_{\zz}\left(\frac{1}{\theta-t}
\left[\lambda_1,\lambda_2\right]
(\Psi_{\phi}^{\tr})^{(-1)}\right).
\]\end{itemize}
\end{theorem}
\begin{proof} We prove the first assertion. Let $L(\phi/A)$ be the Taelman $L$-value corresponding to $\phi$. Then by \cite[Thm. 1]{Tae12}, we know that there exist $\mathfrak{f}_{\phi}\in K_{\infty}^{\times}$ and $\mathfrak{c}_1\in A\setminus \{0\}$ such that 
\[
L(\phi/A)=\mathfrak{c}_1\mathfrak{f}_{\phi}.
\]    
We remind the reader that here, $\mathfrak{c}_1$ is the monic generator of the \textit{class module $H(\phi/A)$ of $\phi$} in the sense of Taelman  and $\mathfrak{f}_{\phi}$ is the monic generator of the $A$-module $U(\phi/A)$, so called \textit{the unit module of $\phi$}, consisting of all the elements $x\in K_{\infty}$ satisfying $\Exp_{\phi}(x)\in A$. Moreover, by \cite[Rem. 5]{Tae12}, we know that $L(\phi^{\vee},0)$ differs from $L(\phi/A)$ by a constant in $K^{\times}$. More precisely, this constant, say $\mathfrak{c}_2\in K^{\times}$, is equal to the finite product $ \prod_{v} P_v^{\vee}(1)\inv $, where $v$ runs over all the monic irreducible polynomials in $A$ where $\phi$ has \textit{bad reduction}. Thus we write
\[
L(\phi^{\vee},0)=\mathfrak{c}_1\mathfrak{c}_2\mathfrak{f}_{\phi}.
\]
 On the other hand, since $\Exp_{\phi}$ is a locally isometric map near $0$ with its inverse $\Log_{\phi}$, there exists $N(\phi)\in \mathbb{Z}_{\geq 0}$ such that $\mathfrak{f}_{\phi}/\theta^{N(\phi)}$ lies in the domain of convergence of $\Log_{\phi}$. Therefore, we write
\[
\mathfrak{f}_{\phi}=\theta^{N(\phi)}\Log_{\phi}\left(\Exp_{\phi}\left(\frac{\mathfrak{f}_{\phi}}{\theta^{N(\phi)}}\right)\right).
\]
Now letting $\mathfrak{c}_{\phi}:=\mathfrak{c}_1\mathfrak{c}_2\theta^{N(\phi)}$ and $\zz=\theta^{\mathfrak{m}_{\phi}}\Exp_{\phi}\left(\frac{\mathfrak{f}_{\phi}}{\theta^{N(\phi)}}\right)$ as well as using Theorem \ref{T:Intro log formula}, we obtain 
\begin{multline*}
L(\phi^{\vee},0)=\mathfrak{c}_1\mathfrak{c}_2\mathfrak{f}_{\phi}=\mathfrak{c}_1\mathfrak{c}_2\theta^{N(\phi)} \Log_{\phi}\left(\frac{1}{\theta^{\mathfrak{m}_{\phi}}}\theta^{\mathfrak{m}_{\phi}}\Exp_{\phi}\left(\frac{\mathfrak{f}_{\phi}}{\theta^{N(\phi)}}\right)\right)\\
=\mathfrak{c}_{\phi}\mathcal{M}_{\zz}\left(
\left[\Log_{\phi}\left(\frac{f_{\lambda_1}(t)}{\theta^{m_{\phi}}}\right),\dots,\Log_{\phi}\left(\frac{f_{\lambda_r}(t)}{\theta^{m_{\phi}}}\right)\right]
(\Psi_{\phi}^{\tr})^{(-1)}\right)
\end{multline*}
as desired. We now prove the second assertion. By \cite[Thm. 1.7.5]{Luc25}, we know that
\begin{equation}\label{E:lfuncformula}
L(\phi^{\vee},0)=\log_{\phi}(1).
    \end{equation}
    That is $N(\phi)=0$ and $\mathfrak{f}_{\phi}=\log_{\phi}(1)$.
  Moreover $\mathfrak{m}_{\phi}=0$. Indeed, by \cite[Thm. 4.4]{KP23}, there exists a positive integer $\mathfrak{N}$ such that, for $i=1,2$, one can choose $\lambda_i=\theta^{\mathfrak{N}}\Log_{\phi}(\xi_i)$ for some $\xi_i\in \mathbb{C}_{\infty}^{\times}$. By our condition on $k_1$ and $k_2$ as well as \cite[Thm. 5.2]{KP23}, we have $\mathfrak{N}=1$. Thus, we can choose $\lambda_i=\theta\Log_{\phi}(\xi_i)$. Since $\Exp_{\phi}$ is an isometric map on the domain of convergence of $\Log_{\phi}$, we see that 
    $
    \Exp_{\phi}(\theta^{-j-1}\lambda_j)=\Exp_{\phi}(\theta^{-j}\Log_{\phi}(\xi_j))$ lies in the domain of convergence of $\Log_{\phi}$ for all $j\geq 0$, implying that $\mathfrak{m}_{\phi}=0$. Thus, $\zz=1$ and by the last part of Theorem \ref{T:Intro log formula}, we obtain the second assertion.
    \end{proof}

\subsection{Drinfeld modular forms of arbitrary rank and modular matrix-valued functions}\label{S:DMF}
In this section we apply our formulas of the previous sections to Drinfeld modular forms of arbitrary rank. This enables us to give a connection between certain values of modular matrix-valued functions and values of Goss $L$-functions connected to Drinfeld modules.

We begin recounting the theory of Drinfeld modular forms in some generality, then we specify to the situation we consider in our formulas. For more details on Drinfeld modular forms of arbitrary rank, we refer the reader to \cite{BB17,BBP18,Gek17,Gek25}.

When $r=1$, we set $\mathbb{H}^1:=\{1\}$ and for $r\geq 2$, we recall the Drinfeld upper half plane $\mathbb{H}^r$ given by
\[
\mathbb{H}^r=\mathbb{P}^{r-1}(\CC_{\infty})\setminus \{K_{\infty}\text{-rational hyperplanes}\}
\]
so that any element $\ww=(w_1,\dots,w_r)^{\tr}\in \mathbb{H}^r$ is given by $K_{\infty}$-linearly independent entries and normalized so that $w_r=1$. Let $b_1,\dots,b_r\in K_{\infty}$. We define 
\[
\ell_{b_1,\dots,b_r}(\ww):=b_1w_1+\dots+b_rw_r
\]
and let $|\ww|_{\infty}:=\max\{|w_1|,\dots,|w_r|\}$.  
    For any $n\in \mathbb{Z}_{\geq 1}$, we consider \[
    \mathbb{H}_{n}^r:=\{\ww\in \mathbb{H}^r \ \ | \ \ |\ell_{b_1,\dots,b_r}(\ww)|\geq q^{-n}|\ww|_{\infty}, \ \ b_1,\dots,b_r\in K_{\infty}\} \subset \mathbb{H}^r.
    \]
   We remark that $\{\mathbb{H}_{n}^r\}_{n=1}^{\infty}$ is an admissible covering of $\mathbb{H}^r$ \cite[\S3]{BBP18} and hence $\mathbb{H}^r$ is equipped with a rigid analytic space structure. We further call $f:\mathbb{H}^r\to \mathbb{T}$ \textit{a rigid analytic function} if its restriction to each $\mathbb{H}_n^r$ is the uniform limit of rational functions in $\mathbb{T}(w_1,\dots,w_{r-1})$ with no pole in $\mathbb{H}_n^r$. Furthermore, we call $f:\mathbb{H}^r\to \Mat_{m\times n}(\mathbb{T})$ \textit{a rigid analytic matrix-valued function} if its each entry is a rigid analytic function.

For any $\gamma=(a_{ij})\in \GL_r(K_\infty)$, define
\[
\gamma \cdot \ww:=\Big(\frac{a_{11}w_1+\dots+ a_{1r}w_r}{a_{r1}w_1+\dots+ a_{rr}w_r},\dots,\frac{a_{(r-1)1}w_1+\dots +a_{(r-1)r}w_r}{a_{r1}w_1+\dots+ a_{rr}w_r},1\Big)^{\tr}\in \mathbb{H}^r.
\]
We further set
\[
j(\gamma,\ww)=a_{r1}w_1+\dots+ a_{rr}w_r\in \CC_{\infty}^{\times}.
\]
For any $\ww=(w_1,w_2,\dots,w_r)^{\tr}\in \mathbb{H}^r$, we let $\widetilde{\ww}:=(w_2,\dots,w_r)^{\tr}\in \mathbb{H}^{r-1}$ and consider
\[
u(\ww):=\frac{1}{\tilde{\pi}\Exp_{\phi(\widetilde{\ww})}(w_1)}.
\]
Here, by $\Exp_{\phi(\widetilde{\ww})}$, we mean the exponential function of the Drinfeld module $\phi(\widetilde{\ww})$ of rank $r-1$ given by 
\[(\phi(\widetilde{\ww}))_{\theta} := \theta +\widetilde{g}_1(\widetilde{\ww})\tau + \cdots+\widetilde{g}_r(\widetilde{\ww})\tau^{r-1}\]
which corresponds to the $A$-lattice of rank $r-1$ generated by the $K_{\infty}$-linearly independent elements $w_2,\dots,w_r$.

In what follows, we define 
\[
|\ww|_{\text{im}}:=\mathrm{inf}\{\inorm{w_1-\mathfrak{a}}:\mathfrak{a}=a_2w_2+\dots+a_rw_r, \ \  a_2,\dots,a_r\in K_{\infty}\}
\]
and following the notation in \cite[\S5.1]{CG22}, for $\widetilde{\ww}\in \mathbb{H}^{r-1}$, we set 
\[
\mathfrak{I}_{\widetilde{\ww}}:=\{ \ww=(w_1,\dots,w_r)^{\tr}=:(w_1,\widetilde{\ww})\in \mathbb{H}^{r} \ \ | \ \ \inorm{w_1}=|\ww|_{\text{im}} \}.
\]

\begin{definition}\label{D:Modular forms def}
We call an analytic function $f:\HH^{r} \to \mathbb{T}$  \textit{modular-like of weight $\ell\in \mathbb{Z}$ and type $m\in \Z/(q-1)\Z$} if for any $\gamma  \in \GL_r(A)$, we have
\[f(\gamma\cdot \ww) = j(\gamma,\ww)^{\ell}\det(\gamma)^{-m} f(\ww).\]
 We say that a modular-like function $f$ is
\begin{enumerate}
\item \textit{weakly modular} if, for any choice of $\widetilde{\ww}\in \mathbb{H}^{r-1}$, there exists a non-negative integer $M$ such that 
\[
\lim_{\substack{\ww=(w_1,\widetilde{\ww})\in \mathfrak{I}_{\widetilde{\ww}}\\ |\ww|_{\infty}\to\infty}}u(\ww)^Mf(\ww)<\infty,
\]
\item \textit{a modular form} if 
\[
\lim_{\substack{\ww=(w_1,\widetilde{\ww})\in \mathfrak{I}_{\widetilde{\ww}}\\ |\ww|_{\infty}\to\infty}}f(\ww)<\infty,
\]
\item \textit{a cusp form} if 
\[
\lim_{\substack{\ww=(w_1,\widetilde{\ww})\in \mathfrak{I}_{\widetilde{\ww}}\\ |\ww|_{\infty}\to\infty}}f(\ww)=0.
\]
\end{enumerate}
\end{definition}

By \cite[Cor.3.6]{GU25}, we know that modular forms in our context are closely related to $\mathbb{C}_{\infty}$-valued Drinfeld modular forms. Hence, every modular form $f(\ww)$ has an expansion as a power series in $u(\ww)$.  More precisely, in some neighborhood of infinity, one can uniquely determine $f$ via the power series
\[
f(\ww) = \sum_{i=0}^\infty f_i(\widetilde{\ww}) u(\ww)^i,
\]
which is called \textit{the $u$-expansion of $f$}. Here each $f_i:\mathbb{H}^{r-1}\to \mathbb{T}$ is a uniquely determined rigid analytic function. We note that when $r=2$, each $f_i$ is a constant in $\mathbb{T}$. This is akin to a Fourier expansion, with the key difference that we lack a fundamental way to compute the $u$-expansion coefficients of $f$.
We wish to generalize the above definition to encompass matrix-valued deformations of modular forms. Our definition captures the notion of vector-valued modular forms defined by Pellarin and Pellarin and Perkins in the rank two setting for a particular representation (see \cite{Pel25,PP18}).

Recall from \S\ref{S:Intro1.3} the identity representation that sends $\gamma\to \overline \gamma$ and the representation $\rho^{*}$ that sends $\gamma \to (\overline \gamma^{\tr})^{-1}$.
Inspired by the work of Pellarin \cite{Pel12} and Pellarin and Perkins \cite{PP18}, we introduce the following rigid analytic functions.
\begin{definition}\label{D:matrix modular type}
\begin{itemize}
\item[(i)] A rigid analytic matrix-valued function $f:\mathbb{H}^r \to \Mat_{r}(\TT)$ is \textit{a modular-like matrix-valued function  with weights $(\ell_1,\dots,\ell_r) \in \Z^r$ and type $~m\in \Z/(q-1)\Z$ with respect to $\rho^{*}$} if for all $\gamma\in \GL_r(A)$, we have
\[f(\gamma\cdot \ww) =\det(\gamma)^{-m} \rho^{*}(\gamma) f(\ww) \begin{pmatrix}
    j(\gamma,\ww)^{\ell_1}& & & \\
     & \ddots & &  \\
      & & & j(\gamma,\ww)^{\ell_r}
      \end{pmatrix}.\]
\item[(ii)]  We say that a modular-like matrix valued function with weights $(\ell_1,\cdots,\ell_r) \in \Z^r$ and type $m\in \Z/(q-1)\Z$ with respect to $\rho^{*}$ is \textit{modular} if, for any $\widetilde{\ww}\in \mathbb{H}^{r-1}$, we have
\[
\begin{pmatrix}
    1& & & & \\
     & u(\ww) & & & \\
      & & & \ddots& \\
      & & & & u(\ww)
      \end{pmatrix}f(\ww)\to \Mat_{r}(\mathbf{0})
\]
as $|\ww|_{\infty}\to \infty$ where $\ww=(w_1,\widetilde{\ww}) \in \mathfrak{I}_{\widetilde{\ww}}$.
\end{itemize}
\end{definition}
We comment that, in the rank two case, if $f$ is a modular matrix valued function, then the each column of $f$ is necessarily 
\textit{a $\mathbb{T}$-valued vectorial modular form of weight $\ell_i$ and type $m$} in the sense of \cite[Def. 3.4]{PP18}.

\begin{definition}
If a function $f:\mathbb{H}^r \to \Mat_{r}(\TT)$ is such that $f\twistk{k}$ is a modular  matrix valued function for some value $k\in \Z_+$ with $k$ minimal of weights $(q^k\ell_1,\dots,q^k\ell_r)$ and type $m$, then we say that $f$ is \textit{a modular matrix valued function with weights $(\ell_1,\dots,\ell_r)$, type $m$ and root $q^k$}.
\end{definition}

Let $\ww\in \HH^r$ and $\Lambda_{\ww}=Aw_1+\cdots+Aw_r$  be an $A$-lattice of rank $r$ in $\C_\infty$ generated by $w_1,\dots,w_r$. One can consider a Drinfeld module $\phi(\ww)$ of rank $r$ corresponding to $\Lambda_{\ww}$ so that it is given by 
\[
(\phi(\ww))_{\theta} := \theta + g_1(\ww)\tau + \cdots+g_r(\ww)\tau^r.
\]
If we vary $\ww\in \mathbb{H}^r$ above, by \cite[\S15]{BBP18}, we see that $g_i:\mathbb{H}^r\to \mathbb{C}_{\infty}$ is a modular form of weight $q^i-1$ and type $0$ for $1\leq i \leq r$. In addition, the rigid analytic function $g_r:\mathbb{H}^r\to \mathbb{C}_{\infty}$ is a cusp form of weight $q^r-1$ and type $0$. Furthermore, in \cite{Gek17}, Gekeler defined a non-zero type Drinfeld modular form $h_r$ (when $r=1$, it is set to be $h_1:=-1$) so that 
\begin{equation}\label{E:hfunction}
g_r(\ww)=(-1)^{r-1}h_r(\ww)^{q-1}.
\end{equation}
More precisely, $h_r$ is a cusp form of weight $q^r-1/(q-1)$ and type $1$. 

With these definitions in place, we revisit the constructions of \S\ref{S:Drinfeld} 
from the viewpoint of Drinfeld modular forms. For each $1\leq i \leq r$, we denote by $f_i(\cdot,t):\mathbb{H}^r\to \mathbb{T}$ the rigid analytic function that sends each $\bw\in \mathbb{H}^r$ to the Anderson generating function of $\phi(\ww)$ with respect to $w_i$. Moreover, consider matrix valued rigid analytic function $\Upsilon:\mathbb{H}^r\to \GL_{r}(\mathbb{T})$ defined by (note that in \cite{Pel14} it is called $\hat{\Psi}$ in the rank two case)
\[\Upsilon(\ww) := \begin{pmatrix} f_{1}(\ww,t) & \cdots & \cdots & f_{r}(\ww,t)\\
f_{1}(\ww,t)^{(1)} & \cdots & \cdots & f_{r}(\ww,t)^{(1)}\\
\vdots  & & & \vdots\\
\vdots  & & & \vdots\\
f_{1}(\ww,t)^{(r-1)} & \cdots & \cdots &f_{r}(\ww,t)^{(r-1)}
\end{pmatrix}.\]
In the proof of \cite[Thm. 5.5]{CG22} (see also \cite[Lem. 2.4]{Pel14} in the rank two case), it is shown that for any $\gamma \in \GL_r(A)$, we have
\begin{equation}\label{E:Upsilon transform}
\Upsilon(\gamma\cdot \ww)=\begin{pmatrix}
    j(\gamma,\ww)^{-1}& & & \\
     & \ddots & &  \\
      & & & j(\gamma,\ww)^{-q^{r-1}}
      \end{pmatrix}\Upsilon(\ww) \overline \gamma^{\tr}.
\end{equation}
Further, by \cite[Prop. 3.4]{CG22} (see also \cite[Lem. 2.3]{Pel14} in the rank two case), we have 
\begin{equation}\label{E:determinant}
\det(\Upsilon(\ww))=\frac{\omega_C}{h_r(\ww)}.
\end{equation}
Here, for clarity, comparing to \cite[Prop. 3.4]{CG22}, the factor of $\tilde{\pi}$ is avoided due to our normalization in \eqref{E:hfunction} (compare with \cite[(2.11)]{CG22}).
 
Next, we consider the rigid analytic function function $V:\mathbb{H}^r\to \GL_r(\mathbb{C}_{\infty})$ defined by
\[
V(\ww):=\begin{pmatrix}
g_1(\ww)&g_2(\ww)^{(-1)} &g_3(\ww)^{(-2)}& \dots &g_r(\ww)^{(1-r)}\\
\vdots &\vdots&\vdots&  \iddots & \\
\vdots & \vdots& g_r(\ww)^{(-2)} & & \\
\vdots &g_r(\ww)^{(-1)} & & & \\
g_r(\ww) &  & & & 
\end{pmatrix}, \ \ \ww\in \mathbb{H}^r.
\]
Following \S\ref{S:Drinfeld}, we then define the function $\Psi:\mathbb{H}^r\to \GL_r(\mathbb{T})$ given by 
\begin{equation}\label{E:ratmod}
\Psi(\ww) := ((\Upsilon(\ww)\twist)^{\tr} V(\ww))\inv.
\end{equation} 
It is clear that $\Psi$ is a matrix valued rigid analytic function. From the transformation property of $\Upsilon(\ww)$, we deduce the following.

\begin{proposition}\label{P:modlike}
The matrix valued function $\ww\to ((\Psi(\ww)^{\tr})^{(-1)}$ is  modular of weights $(1/q,\dots,1/q^r)$, type $0$ and root $q^r$ with respect to $\rho^{*}$. In other words, for all $\gamma\in \GL_r(A)$,  it satisfies
\begin{equation}\label{E:weakmod0}
((\Psi(\gamma \cdot \ww)^{\tr})^{(r-1)}=\rho^{*}(\gamma)(\Psi(\ww)^{\tr})^{(r-1)}\begin{pmatrix}
    j(\gamma,\ww)^{q^{r-1}}& & & & \\
     & \ddots & & & \\
      & & & j(\gamma,\ww)^{q}& \\
      & & & & j(\gamma,\ww)
      \end{pmatrix}.
      \end{equation}
Moreover, we have 
\[\det(\Psi(\ww)^{(r-1)}) = (t-\theta^{q^{r-1}})^{-1}\cdots (t-\theta^q)^{-1} h_r(\ww) \Omega(t).\]
\end{proposition}

\begin{proof} First we note that the last assertion simply follows from \eqref{E:determinant} and the fact that $(\omega_C^{(r)})^{-1}=(t-\theta^{q^{r-1}})^{-1}\cdots (t-\theta^q)^{-1} \Omega(t)$. Now we prove the first assertion. Observe that 
\begin{multline*}
((\Psi(\ww)^{\tr})^{(r-1)})^{-1}=(V(\ww)^{\tr})^{(r-1)}\Upsilon(\ww)^{(r)}\\=
\begin{pmatrix}
g_1(\ww)^{(r-1)}&g_2(\ww)^{(r-1)} & \dots & \dots &g_r(\ww)^{(r-1)}\\
\vdots &\vdots&&  \iddots & \\
\vdots & \vdots& \iddots & & \\
\vdots &g_r(\ww)^{(1)} & & & \\
g_r(\ww) &  & & & 
\end{pmatrix}\begin{pmatrix} f_{1}(\ww,t)^{(r)} & \cdots & \cdots & f_{r}(\ww,t)^{(r)}\\
f_{1}(\ww,t)^{(r+1)} & \cdots & \cdots & f_{r}(\ww,t)^{(r+1)}\\
\vdots  & & & \vdots\\
\vdots  & & & \vdots\\
f_{1}(\ww,t)^{(2r-1)} & \cdots & \cdots &f_{r}(\ww,t)^{(2r-1)}
\end{pmatrix}.
\end{multline*}
For $1\leq i,j\leq r$, letting $g_i\equiv 0$ when $i\leq 0$ and $i>r$, we set
\[
\mathfrak{F}_{ij}(\ww,t):=g_i^{(r-i)}(\ww)f_j^{(r)}(\ww,t)+g_{i+1}^{(r-i)}(\ww)f_j^{(r+1)}(\ww,t) +\cdots+g_r^{(r-i)}(\ww)f_j^{(2r-i)}(\ww,t)
\]
which is the $(i,j)$-th entry of $ ((\Psi(\ww)^{\tr})^{(r-1)})^{-1}$.
By \eqref{E:Upsilon transform} and the modularity properties of $g_1,\dots,g_r$, for $1\leq i,j\leq r$, we see that 
\[
\mathfrak{F}_{ij}(\gamma\cdot \ww,t)=j(\gamma,\ww)^{-q^{r-i}}(a_{j1}(t)\mathfrak{F}_{j1}(\ww,t)+\cdots +a_{jr}(t)\mathfrak{F}_{jr}(\ww,t)), \ \ \ \ \ \ \ \gamma=(a_{ij})_{i,j}\in \GL_r(A).
\]
In other words, we obtain 
\[
((\Psi(\gamma \cdot \ww)^{\tr})^{(r-1)})^{-1}=\begin{pmatrix}
    j(\gamma,\ww)^{-q^{r-1}}& & & & \\
     & \ddots & & & \\
      & & & j(\gamma,\ww)^{-q}& \\
      & & & & j(\gamma,\ww)^{-1}
      \end{pmatrix}(\Psi(\ww)^{\tr})^{(r-1)})^{-1}\overline{\gamma}^{\tr}
\]
implying that 
\[
(\Psi(\gamma \cdot \ww)^{\tr})^{(r-1)}=(\overline{\gamma}^{\tr})^{-1}(\Psi(\ww)^{\tr})^{(r-1)}\begin{pmatrix}
    j(\gamma,\ww)^{q^{r-1}}& & & & \\
     & \ddots & & & \\
      & & & j(\gamma,\ww)^{q}& \\
      & & & & j(\gamma,\ww)
      \end{pmatrix}.
\]
Thus, we see that the function $\ww\to ((\Psi(\ww)^{\tr})^{(-1)}$ is a modular-like matrix-valued function of weights $(1/q,1/q^2,\dots,1/q^r)$, type $0$ and root $q^r$ with respect to $\rho^{*}$. 

In what follows, we show that the rigid analytic matrix valued function $\ww\to ((\Psi(\ww)^{\tr})^{(r-1)}$ is indeed modular. Using Lemma \ref{L:AGF}(iii), for $1\leq i \leq r$, one can write
\[
\mathfrak{F}_{ij}(\ww,t)=(t-\theta^{q^{r-i}})f_j(\ww,t)^{(r-i)}-g_1(\ww)^{(r-i)}f_j(\ww,t)^{(r-i+1)}-\cdots-g_{i-1}(\ww)^{(r-i)}f_j(\ww,t)^{(r-1)}.
\]
Let $\widetilde{\ww}\in \mathbb{H}^{r-1}$ and let $\Lambda_{\widetilde{\ww}}$ be the $A$-lattice generated by the entries of $\widetilde{\ww}$ over $A$. Recall the Drinfeld module $\phi(\widetilde{\ww})$ of rank $r$ corresponding to the $A$-lattice $\Lambda_{\widetilde{\ww}}$ so that it is given by 
\[
(\phi(\widetilde{\ww}))_{\theta} = \theta + \widetilde{g}_1(\ww)\tau + \cdots+\widetilde{g}_{r-1}(\ww)\tau^{r-1}.
\] By \cite[Lem. 3.15]{GU25}, we know that 
 \begin{equation}\label{E:limit0}
 \lim_{\substack{\ww=(w_1,\widetilde{\ww})\in \mathfrak{I}_{\widetilde{\ww}}\\ |\ww|_{\infty}\to\infty}} u(\ww)f_1(\ww,t)^{(i)}= \begin{cases}
 0 &\text{ if } 1\leq i \leq r-2\\
\tilde{\pi}^{-1}\widetilde{g}_{r-1}(\widetilde{\ww})^{-1} &\text{ if } i=r-1.
\end{cases}
\end{equation}
Thus, for $1\leq i \leq r-1$, we obtain
\begin{equation}\label{E:cof00}
   \lim_{\substack{\ww=(w_1,\widetilde{\ww})\in \mathfrak{I}_{\widetilde{\ww}}\\ |\ww|_{\infty}\to\infty}}u(\ww)\mathfrak{F}_{i1}(\ww,t)=(t-\theta^{q^{r-i}})\tilde{\pi}^{-1}\widetilde{g}_{r-1}(\widetilde{\ww})^{-1}
\end{equation}
and 
\begin{equation}\label{E:cof000}
   \lim_{\substack{\ww=(w_1,\widetilde{\ww})\in \mathfrak{I}_{\widetilde{\ww}}\\ |\ww|_{\infty}\to\infty}}u(\ww)\mathfrak{F}_{r1}(\ww,t)=\lim_{\substack{\ww=(w_1,\widetilde{\ww})\in \mathfrak{I}_{\widetilde{\ww}}\\ |\ww|_{\infty}\to\infty}}u(\ww)g_r(\ww)f_1(\ww,t)^{(r)}=\tilde{\pi}^{-1}\widetilde{g}_{r-1}(\widetilde{\ww})^{q-1}
\end{equation}
where the last line of the above equality follows from \cite[Cor. 6.3]{BB17} and \eqref{E:hfunction}.

Note that for $1\leq i \leq r-1$,  by \cite[Prop. 15.12(c)]{BBP18}, we have 
\[
\lim_{\substack{\ww\in \mathfrak{I}_{\widetilde{\ww}}\\ |\ww|_{\infty}\to\infty}}g_{i}(\ww)=\widetilde{g}_{i}(\widetilde{\ww}).
\]
 For $2\leq j \leq r$, by using Lemma \ref{L:AGF}(i), we write
    \[
    f_{j}(\ww,t)=\sum_{i=0}^{\infty}\frac{\alpha_i(\ww)w_j^{q^i}}{\theta^{q^i}-t}
    \]
    where $\alpha_i(\ww)$ is the $i$-th coefficient of the exponential series $\Exp_{\phi(\ww)}$. By \cite[Prop. 15.3(b,c)]{BBP18}, we know that each $\alpha_i$ is a modular form of weight $q^i-1$ and type $0$. Moreover, they admit a $u$-expansion with the constant term $\widetilde{\alpha}_i(\widetilde{\ww})$ which is the $i$-th coefficient of the exponential series $\Exp_{\phi(\widetilde{\ww})}$. This implies that, for any fixed choice of $\widetilde{\ww}\in \mathbb{H}^{r-1}$, we have
    \[
    \lim_{\substack{\ww=(w_1,\widetilde{\ww})\in \mathfrak{I}_{\widetilde{\ww}}\\ |\ww|_{\infty}\to\infty}}f_j(\ww,t)= \sum_{i=0}^{\infty}\frac{\widetilde{\alpha}_i(\widetilde{\ww})w_j^{q^i}}{\theta^{q^i}-t}<\infty.
    \]
        Thus, for $2\leq j \leq r$ and $1\leq i \leq r-1$ we obtain 
        \begin{equation}\label{E:cof01}
        \lim_{\substack{\ww=(w_1,\widetilde{\ww})\in \mathfrak{I}_{\widetilde{\ww}}\\ |\ww|_{\infty}\to\infty}}\mathfrak{F}_{ij}(\ww,t)<\infty
        \end{equation}
        and 
        \begin{equation}\label{E:cof11}
        \mathfrak{F}_{rj}(\ww,t)=g_{r}(\ww)f_j(\ww,t)^{(r)}=\widetilde{g}_{r-1}(\widetilde{\ww})^qu(\ww)^{q-1}+O(u(\ww)^q)
        \end{equation}
        where, again, the last line of the above equality follows from \cite[Cor. 6.3]{BB17} and \eqref{E:hfunction}.

Lastly, by \cite[Cor. 6.3]{BB17}, we observe that 
\begin{multline}\label{E:limit3}
\lim_{\substack{\ww=(w_1,\widetilde{\ww})\in \mathfrak{I}_{\ww}\\ |\ww|_{\infty}\to\infty}}\det(V(\ww)^{(r-1)}\Upsilon^{(r)}(\ww))^{-1}= \lim_{\substack{\ww=(w_1,\widetilde{\ww})\in \mathfrak{I}_{\ww}\\ |\ww|_{\infty}\to\infty}}\det(\Psi(\ww)^{(r-1)}) \\
=\lim_{\substack{\ww=(w_1,\widetilde{\ww})\in \mathfrak{I}_{\ww}\\ |\ww|_{\infty}\to\infty}}(t-\theta^{q^{r-1}})^{-1}\cdots (t-\theta^q)^{-1} h_r(\ww)\Omega(t)\\
=(-1)^rh_{r-1}(\widetilde{\ww})^q (t-\theta^{q^{r-1}})^{-1}\cdots (t-\theta^q)^{-1}\Omega(t) u(\ww)+O(u(\ww)^2).
\end{multline}
Now we write 
\[
(\Psi(\ww)^{\tr})^{(r-1)}=(V(\ww)^{(r-1)}\Upsilon(\ww)^{(r)})^{-1}=\frac{1}{\det(V(\ww)^{(r-1)}\Upsilon^{(r)}(\ww))}\Cof(\ww)^{\tr}=(\mathfrak{C}_{ij})_{i,j}
\]
where $\Cof(\ww)\in \GL_r(\mathbb{T})$ is the cofactor matrix of $V(\ww)^{(r-1)}\Upsilon(\ww)^{(r)}$. Using \eqref{E:cof00}--\eqref{E:limit3}, we see that
         \begin{equation}\label{E:id11}
        \mathfrak{C}_{1r}(\ww)=\mathfrak{c}_{1r}(\widetilde{\ww})u(\ww)+O(u(\ww)^{2})
        \end{equation}
      and   for $1\leq i\leq r-1$, we have 
        \begin{equation}\label{E:id12}
        \mathfrak{C}_{1i}(\ww)=\mathfrak{c}_{1i}(\widetilde{\ww})u(\ww)^q+O(u(\ww)^{q+1})
        \end{equation}
        for some rigid analytic functions $\mathfrak{c}_{1i}:\mathbb{H}^{r-1}\to \mathbb{T}$ and $\mathfrak{c}_{1r}:\mathbb{H}^{r-1}\to \mathbb{T}$. 
Similarly, for $1\leq j \leq r$ and $2\leq i \leq r$, we obtain 
  \begin{equation}\label{E:id123}
        \mathfrak{C}_{ij}(\ww)=\mathfrak{c}_{ij}(\widetilde{\ww})+O(u(\ww))
        \end{equation}
        for some rigid analytic function $\mathfrak{c}_{ij}:\mathbb{H}^{r-1}\to \mathbb{T}$. Since,
        \[
        \lim_{\substack{\ww=(w_1,\widetilde{\ww})\in \mathfrak{I}_{\ww}\\ |\ww|_{\infty}\to\infty}}u(\ww)= \lim_{\substack{\ww=(w_1,\widetilde{\ww})\in \mathfrak{I}_{\ww}\\ |\ww|_{\infty}\to\infty}}\frac{1}{\Exp_{\phi(\widetilde{\ww})}(\tilde{\pi}w_1)}=0,
        \]
we obtain 
\[
\begin{pmatrix}
    1& & & & \\
     & u(\ww) & & & \\
      & & & \ddots& \\
      & & & & u(\ww)
      \end{pmatrix}(\Psi(\ww)^{\tr})^{(r-1)}\to \Mat_r(\mathbf{0})
      \]
     as $|\ww|_{\infty}\to \infty$ where $\ww=(w_1,\widetilde{\ww}) \in \mathfrak{I}_{\widetilde{\ww}}$ for any fixed choice of $\widetilde{\ww}\in \mathbb{H}^{r-1}$. Hence, the function  $\ww\to ((\Psi(\ww)^{\tr})^{(r-1)}$ is modular, finishing the proof of the proposition.

\end{proof}

We now restrict our attention to the specific application for our paper. Let $\ww\in \mathbb{H}^r$. Let $\Lambda_{\ww}$ and $\phi(\ww)$ be as above. By Drinfeld \cite{Dri74}, we know that the moduli space of isomorphism classes of Drinfeld modules of rank $r$ defined over $\mathbb{C}_{\infty}$ is the quotient $\GL_r(A)\setminus \HH^r$. Thus, up to passing to an isomorphic Drinfeld module if necessary, we can assume that $\Lambda_{\ww}$ corresponds to $\phi(\ww)$ such that $g_1(\ww),\dots,g_r(\ww)$ are in $A$. 

Our next corollary is a consequence of Theorem \ref{T:modformslvalues}.

\begin{corollary}\label{C:modlike} Let $\ww=(w_1,\dots,w_r)^{\tr}\in \mathbb{H}^{r}$ be such that, up to passing to an isomorphism, its corresponding Drinfeld module $\phi(\ww)$ is defined over $A$.  Let $\mathfrak{f}_{\phi(\ww)}\in K_{\infty}^{\times}$, $\mathfrak{c}_{\phi(\ww)}\in K^{\times}$ and $N(\phi(\ww))\in \mathbb{Z}_{\geq 0}$ be as in Theorem \ref{T:modformslvalues}.
 Set $\bomega:= \theta^{\mathfrak{m}_{\phi(\ww)}}\Exp_{\phi(\ww)}(\mathfrak{f}_{\phi(\ww)}/\theta^{N(\phi(\ww))})$. Then we have
\[
L(\phi(\ww)^{\vee},0)=\mathfrak{c}_{\phi(\ww)}\mathcal{M}_{\bomega}\left(
\left[\Log_{\phi(\ww)}\left(\frac{f_{1}(\ww,t)}{\theta^{\mathfrak{m}_{\phi(\ww)}}}\right),\dots,\Log_{\phi(\ww)}\left(\frac{f_{r}(\ww,t)}{\theta^{\mathfrak{m}_{\phi(\ww)}}}\right)\right]
(\Psi(\ww)^{\tr})^{(-1)}\right).
\] 
Moreover, assume that $r=2$ and  $\deg_{\theta}g_i(\ww)<q^{i}$ for $i=1,2$. Then $\bomega=1$ and 
\[
L(\phi(\ww)^{\vee},0)=\mathfrak{c}_{\phi(\ww)}\mathcal{M}_{\bomega}\left(\frac{1}{\theta-t}
\left[w_1,w_2\right]
(\Psi(\ww)^{\tr})^{(-1)}\right).
\]
\end{corollary}



\end{document}